\documentclass{article}

\usepackage{amsmath,amssymb,euscript,mathrsfs,amsthm}
\usepackage{graphicx}
\usepackage{caption}
\usepackage{subcaption}
\usepackage[utf8]{inputenc}
\usepackage[T2A]{fontenc}
\usepackage[russian,english]{babel}
\usepackage{amssymb}
\usepackage{graphicx}
\usepackage[all]{xy}
\usepackage{color}
\usepackage{latexsym}
\usepackage{amssymb,amsmath,amstext}
\usepackage{epsfig}
\usepackage{graphics}
\usepackage{euscript}
\usepackage{ifthen}
\usepackage{wrapfig}
\usepackage{amsthm}
\usepackage{multicol}
\usepackage{paralist}
\usepackage{verbatim}
\usepackage[margin=1in]{geometry}
\usepackage{algorithmicx}
\usepackage{algpseudocode}
\usepackage{algorithm}
\usepackage{placeins}
\usepackage{caption}
\usepackage{subcaption}
\usepackage{longtable}
\usepackage{lscape}
\usepackage{empheq}

\usepackage{float}

\def\bbR{\mathbb{R}}

\def\cI{{\mathcal I}}

\def\cN{{\mathcal N}}

\def\cS{{\mathcal S}}

\def\Qfn{{\tilde{Q}}}

\numberwithin{equation}{section}
\theoremstyle{plain}
\newtheorem{theorem}{Theorem}
\numberwithin{theorem}{section}
\newtheorem{proposition}[theorem]{Proposition}
\newtheorem{lemma}[theorem]{Lemma}

\newtheorem{definition}[theorem]{Definition}
\newtheorem{conditions}[theorem]{Conditions}

\theoremstyle{remark}

\allowdisplaybreaks

\newcommand{\minimize}{\textnormal{minimize}}
\newcommand{\maximize}{\textnormal{maximize}}

\newcommand{\psf}{\psi}

\newcommand{\signal}{x}
\newcommand{\numc}{M}
\newcommand{\coeff}{c}

\newcommand{\muopt}{\mu_\star}
\newcommand{\dualfn}{Q}

\newcommand{\Lpspace}[1]{L^2_{#1}}
\newcommand{\Lprod}[2]{{\left \langle #1 \right \rangle}_{#2}}

\DeclareMathOperator{\supp}{supp}

\newcommand{\deltafn}[1]{\delta_{#1}}
\newcommand{\defn}{:=}
\newcommand{\dist}{P}

\newcommand{\Kp}{{K_\dist}}

\newcommand{\dy}{\partial_2}
\newcommand{\dx}{\partial_1}
\newcommand{\aie}{\alpha_\epsilon^{[i]}}
\newcommand{\bie}{\beta_\epsilon^{[i]}}

\newcommand{\obs}{P} 
\newcommand{\Amat}{\Lambda}
\newcommand{\Kvec}{\kappa}

\newcommand\itemrow[2]{%
  \item[{\makebox[10em][l]{#1}}]{#2}%
}
\newcommand{\CondMild}{\textsc{Integrability}}
\newcommand{\CondW}{\textsc{Positivity}}
\newcommand{\CondPSD}{\textsc{Independence}}
\newcommand{\CondOnesMatrix}{\textsc{Determinantal}}

\newcommand{\nodal}{\tau}
\newcommand{\nnodal}{z}
\newcommand{\numNod}{N_1}
\newcommand{\numNnod}{N_0}
\newcommand{\extra}{\tau}

\newcommand{\TSet}{\mathcal{T}}

\newcommand{\ddd}{\gamma}

\date{\today}
\author{Geoffrey Schiebinger$^*$, Elina Robeva$^\sharp$ and Benjamin Recht$^*\dagger$\\
 $^*$Deparment of Statistics
 $^\sharp$Department of Mathematics\\
 $^\dagger$Department of Electrical Engineering and Computer Science\\
 University of California, Berkeley}

\begin{document}
\title{Superresolution without Separation}
\maketitle

\begin{abstract}
This paper provides a theoretical analysis of diffraction-limited superresolution, demonstrating that arbitrarily close point sources can be resolved in ideal situations.  
Precisely, we assume that the incoming signal is a linear combination of $M$ shifted copies of a known waveform with unknown shifts and amplitudes, and one only observes a finite collection of evaluations of this signal.  We characterize  properties of the base waveform such that the exact translations and amplitudes can be recovered from $2M+1$ observations.  This recovery is achieved by solving a 
a weighted version of basis pursuit over a continuous dictionary.    Our methods combine classical polynomial interpolation techniques  with contemporary tools from compressed sensing. 
\end{abstract}

\section{Introduction}

Imaging below the diffraction limit remains one of the most practically important yet theoretically challenging problems in signal processing.  Recent advances in superresolution imaging techniques have made substantial progress towards overcoming these limits in practice~\cite{natureMethods,nobelPrize}, but theoretical analysis of these powerful methods remains elusive.  Building on polynomial interpolation techniques and tools from compressed sensing, this paper provides a theoretical analysis of diffraction-limited superresolution, demonstrating that arbitrarily close point sources can be resolved in ideal situations.  

We assume that the measured signal takes the form
\begin{equation}
\label{EqnSignal}
x(s) = \sum_{i=1}^M c_i \psi(s,t_i),
\end{equation}
Here $\psf(s,t)$ is a function that describes the image at spatial location $s$ of a point source of light localized at $t$. The function $\psf$ is called the {\em point spread function}, and we assume its particular form is known beforehand. In~\eqref{EqnSignal},  $t_1, \ldots, t_M$ are the locations of the point sources and $c_1,...,c_M > 0$ are their intensities. Throughout we assume that these quantities together with the number of point sources $M$, are fixed but unknown.  The primary goal of superresolution is to recover the locations and intensities from a set of noiseless observations
$$\{x(s) \; | \; s\in \cS\}\,.$$
A mock-up of such a signal $x$ is displayed in Figure~\ref{fig:mockup}.

\begin{figure}
\centering
\includegraphics[width= \textwidth]{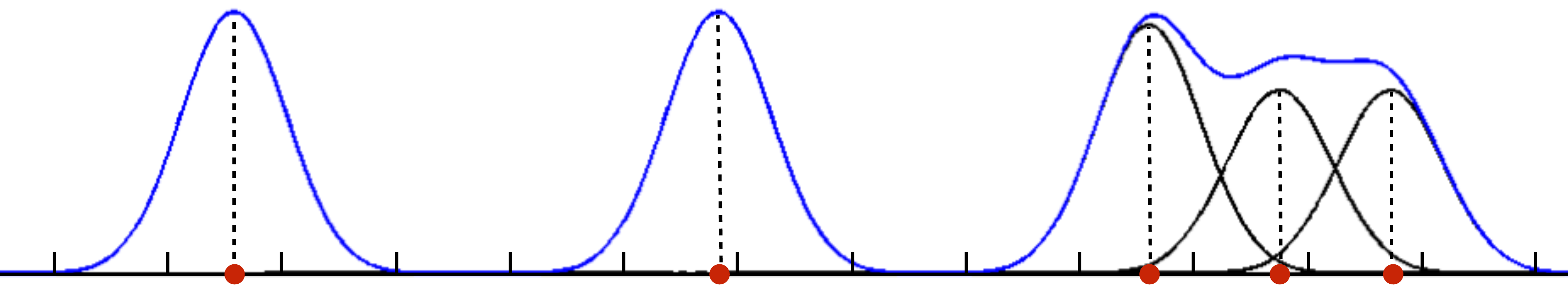}
\caption{An illustrative example of~\eqref{EqnSignal} with the Gaussian point spread function $\psf(s,t)= e^{-(s-t)^2}$.  
The $t_i$ are denoted by red dots, and the true intensities $c_i$ are illustrated by vertical, dashed black lines.  
The super position resulting in the signal $x$ is plotted in blue.  The samples $\cS$ would be observed at the tick marks on the horizontal axis.}
\label{fig:mockup}
\end{figure}

In this paper, building on the work of Cand\`es and Fernadez-Granda~\cite{cg_noisy,CandesGranda,granda2} and Tang~\emph{et al}~\cite{Bhaskar13,Tang14,Tang12}, we aim to show that we can recover the tuple $(t_i, c_i, M)$ by solving a convex optimization problem.  We formulate the superresolution imaging problem as an infinite dimensional optimization over measures.  Precisely, note that
the observed signal can be rewritten as
\begin{equation}
\label{EqnMeasure}
x(s)  = \sum_{i=1}^M c_i \psi(s,t_i) = \int \psf(s,t) d\muopt(t)\,.
\end{equation}
Here, $\muopt$ is the positive discrete measure $\sum_{i=1}^M c_i\delta_{t_i}$, where $\delta_t$ denotes the Dirac measure centered at $t$.  We aim to show that we can recover $\muopt$ by solving the following:
\begin{equation}
\label{Primal}
\begin{aligned}
\underset{\mu}{\minimize} & \quad \int w(t) \mu(dt) \\ 
\textnormal{subject to} & \quad \signal(s) = \int \psf(s,t)d\mu(t), \quad s \in \cS \\
& \quad \supp \mu \subset B\\
& \quad \mu \geq 0\,.
\end{aligned}
\end{equation}
Here, $B$ is a fixed compact set and $w(t)$ is a weighting function that weights the measure at different locations. The optimization problem~\eqref{Primal} is over the set of all positive finite measures~$\mu$ supported on $B$.

The optimization problem~\eqref{EqnMeasure} is an analog of $\ell_1$ minimization over the continuous domain $B$.  Indeed, if we know a priori that the $t_i$ are elements of a finite discrete set $\Omega$, then optimizing over all measures subject to $\supp \mu \subset \Omega$ is precisely equivalent to weighted $\ell_1$ minimization.  This infinite dimensional analog with uniform weights has proven useful for compressed sensing over continuous domains~\cite{Tang12}, resolving diffraction-limited images from low-pass signals~\cite{cg_noisy,granda2,Tang14}, system identification~\cite{SysID}, and many other applications~\cite{CRPW10}.  We will see below that the weighting function essentially ensures that all of the candidate locations are given equal influence in the optimization problem.

Our main result, Theorem~\ref{ThmMain}, establishes that for one-dimensional signals, under rather mild conditions, we can recover $\muopt$ from the optimal solution of~\eqref{Primal}.  Our conditions, described in full-detail below essentially require the observation of at least $2M+1$ samples, and that the set of translates of the point spread function forms a linearly independent set.  
In Theorem~\ref{ThmGauss} we verify that these conditions are satisfied by the Gaussian point spread function for any $M$ source locations with no minimum separation condition.

Boyd, Schiebinger and Recht~\cite{NickGeoffBen} show that the problem~\eqref{Primal} can be optimized to precision $\epsilon$ in polynomial time using a greedy algorithm.  In our experiments in Section~\ref{SecSym}, we use this algorithm to demonstrate that our theory applies, and show that even in multiple dimensions with noise, we can recover very closely spaced point sources. 

\subsection{Main Result}

We restrict our theoretical attention to the one-dimensional case, leaving the higher-dimensional cases to future work.   Let $\psi:\mathbb{R}^2 \rightarrow \mathbb{R}$ be our one dimensional point spread function, with the first argument denoting the position where we are observing the image of a point source located at the second argument.  

For convenience, we will assume that $B = [-T,T]$ for some large scalar $T$. However, our proof will trivially extend to more restricted subsets of the real line. To define the weighting function in the objective of our optimization problem, fix a positive measure $P$ on $\cS$ and set
\begin{equation*}
w(t) = \int  \psf(s,t)d\obs(s)\,.
\end{equation*}
Here, for concreteness one can think of $P$ as the uniform measure on the set $\cS$ consisting of points at which we observe the signal.  Just note that the particular form of $P$ only has limited impact on the conditions of our main theorem.

Before we proceed to state the main result, we need to introduce a bit of notation.
We define the vector valued function $v : \bbR \to \bbR^{2M}$ via 
\begin{equation}\label{eq:v-definition}
v(s) = \begin{bmatrix}   \psf(s,t_1)  & \ldots & \psf(s,t_M)& \frac{d}{dt_1} \psf(s,t_1) &\ldots& \frac{d}{dt_M}\psf(s,t_M) \end{bmatrix}^T\,.
\end{equation}
Let us also define the vector valued function $\kappa: \bbR \to \bbR^{2M}$ via
\begin{equation}\label{eq:kappa-definition}
\kappa(t) = \frac 1 {w(t)} \int \psf(t,s) v(s) dP(s)\,.
\end{equation}
and the matrix-valued function $\Amat : \bbR^{2M+1} \to \bbR^{{2M+1}\times{2M+1}}$ via
\begin{equation}\label{eq:lambda-definition}
\Amat(p_1,\ldots,p_{2M+1}) \defn \begin{bmatrix} \kappa(p_1) & \ldots & \kappa(p_{2M+1}) \\ 1 & \ldots & 1 \end{bmatrix}.
\end{equation}

\begin{conditions}\label{main-conditions} 
  We impose the following four conditions on the point spread function $\psf$:
\begin{enumerate}
\itemrow{\CondMild}{ $\psf(\cdot,t) \in \Lpspace{\obs}$.}
\itemrow{\CondW}{ For all $t\in [-T,T]$ we have $w(t) > 0$.}
\itemrow{\CondPSD}{ The matrix $\int v(s) v(s)^T dP(s)$ is nonsingular.}
\itemrow{\CondOnesMatrix}{ There exists $\rho > 0$ such that for any
$t_i^-,t_i^+ \in (t_i-\rho,t_i+\rho)$, and  $t\in [-T,T]$, \\  the matrix $\Amat\big(t_1^-, t_1^+,\ldots,t_M^-,t_M^+,t\big)$
is nonsingular whenever $t,t_i^-,t_i^+$ are distinct.}
\end{enumerate} 
\end{conditions}

\begin{theorem}
\label{ThmMain}
If $\psf$ satisfies the Conditions~\ref{main-conditions}~and $|\cS| > 2M$, then the true measure $\muopt$ is the unique optimal solution of~\eqref{Primal}.
\end{theorem}

Note that the first three parts of Conditions~\ref{main-conditions} are both rather mild and rather easy to verify.  \CondMild~simply requires the function to be integrable, and this is trivially true for a continuous point spread function and the uniform measure on a discrete set. \CondW~eliminates the possibility that a candidate point spread function could equal zero at all locations---obviously we would not be able to recover the source in such a setting! \CondPSD~is satisfied if for any $s_1,\ldots,s_{2M} \in \cS$, the matrix $\big[v(s_1) \ldots v(s_{2M})\big]$ is invertible.  Such a condition is necessary if we want to recover the amplitudes uniquely assuming we knew the true $t_i$ locations~\emph{a priori}.

The fourth condition, \CondOnesMatrix, looks impenetrable at first glance, and is indeed nontrivial to verify.  As we will see below, this key condition is related to the existence of a canonical ``dual-certificate'' that is used ubiquitously in sparse approximation.  In compressed sensing, this construction is due to Fuchs~\cite{Fuchs}, but its origins lie in the theory of polynomial interpolation developed by Markov and Tchebycheff, and extended by Gantmacher, Krein, Karlin and others (see the survey in Section~\ref{sec:T-systems}).  Importantly, we show it is satisfied by the Gaussian point spread function with no separation condition.  

\begin{theorem}
\label{ThmGauss}
If $|\cS|> 2M$, then the Conditions~\ref{main-conditions}~hold for $\psf(s,t) = e^{-(s-t)^2}$  for any $t_1 < \ldots < t_M$.
\end{theorem}
Before turning to the proofs of these theorems (c.f. Sections~\ref{Sec4} and~\ref{SecProofThmGauss}), we survey the mathematical theory of superresolution imaging.

\subsection{Foundations: Tchebycheff Systems}\label{sec:T-systems}
Our proofs rely on the machinery of Tchebycheff\footnote{Tchebycheff is one among many transliterations from the cyrillic.  Others include Chebyshev, Chebychev, and Cebysev.} systems.   
This line of work originated in the 1884 doctoral thesis of A. A. Markov on approximating the value of an integral $\int_a^b f(x) dx$ from the moments 
\mbox{$\int_a^b xf(x) dx$, \ldots, $\int_a^b x^n f(x)dx$.}
His work formed the basis of the proof by Tchebycheff (who was Markov's doctoral advisor) of the central limit theorem in 1887~\cite{Tchebycheff}.  

\begin{definition}
A set of functions $u_1,\ldots, u_n$ is called a {\em Tchebycheff system} (or {\em T-system}) if for any $t_1 < \ldots < t_n$, the matrix
$$\begin{pmatrix}
u_1(t_1) & \ldots & u_1(t_n) \\
\vdots \\
u_n(t_1) & \ldots & u_n(t_n)
\end{pmatrix}$$ 
is invertible.
\end{definition}

An equivalent definition is that the functions $u_1,...,u_n$ form a T-system if and only if every linear combination $U(t) = a_1u_1(t) + \cdots + a_n u_n(t)$ has at most $n-1$ zeros. 
One natural example of a T-system is given by the functions $1, t, \ldots, t^{n-1}$.  Indeed, a polynomial of order $n-1$ can have at most $n-1$ zeros.  
Equivalently, the Vandermonde determinant 
does not vanish,
$$ \begin{vmatrix} 
1 & 1 & \ldots & 1 \\ 
t_1 & t_2 & \ldots & t_n \\ 
t_1^2 & t_2^2 & \ldots & t_n^2\\
\vdots \\
t_1^{n-1} & t_2^{n-1} & \ldots & t_n^{n-1}
\end{vmatrix}\ne 0,$$
for any $t_1<\ldots<t_n$.  Just as Vandermonde systmes are used to solve polynomial interpolation problems, T-systems allows the generalization of the tools from polynomial fitting to a broader class of nonlinear function-fitting problems.  Indeed, given a T-system $u_1,...,u_n$, a {\em generalized polynomial} is a linear combination $U(t) = a_1u_1(t) + \cdots + a_n u_n(t)$.  The machinery of T-systems provides a basis for understanding the properties of these generalized polynomials.  For a survey of T-systems and their applications in statistics and approximation theory, see~\cite{GantmacherKrein, Karlin, KarlinStudden}.  In particular, many of our proofs are adapted from~\cite{KarlinStudden}, and we call out the parallel theorems whenever this is the case.

\subsection{Prior art and related work}\label{SecRel}

Broadly speaking, superresolution techniques enhance the resolution of a sensing system, optical or otherwise;  
{\em resolution} is the distance at which distinct sources appear indistinguishable.
The mathematical problem of localizing point sources from a blurred signal has applications in a wide array of empirical sciences:
astronomers deconvolve images of stars to angular resolution beyond the Rayleigh limit~\cite{astro},
and biologists capture nanometer resolution images of fluorescent proteins~\cite{PALM, FPALM,STORM,FastStorm}.
Detecting neural action potentials from extracellular electrode measurements is fundamental to experimental neuroscience~\cite{EkandhamTranchinaSimoncelli}, and resolving the poles of a transfer function is fundamental to system identification~\cite{SysID}.
To understand a radar signal, one must decompose it into reflections from different sources~\cite{mahdi}; 
and to understand an NMR spectrum, one must decompose it into signatures from different chemicals~\cite{Tang14}.

The mathematical analysis of point source recovery has a long history going back to the work of Prony~\cite{Prony} who pioneered techniques for estimating sinusoidal frequencies. 
Modern analysis focuses on $\ell_1$ based methods, although this line of thought goes back at least to Carath\'eodory~\cite{Caratheodory1,Caratheodory2}. 
It is tempting to apply the theory of compressed sensing~\cite{Baraniuk,CandesRombergTao,CandesWakin,Donoho} to problem~\eqref{Primal}.
If one assumes the point sources are located on a finite grid and are well separated, then some of the standard models for recovery are valid (e.g. incoherency,  restricted isometry property, or restricted eigenvalue property). 
With this motivation, many authors solve the gridded form of the superresolution problem in practice~\cite{BajwaHaupt, BaraniukSteeghs, MalioutovCetin,DuarteBaraniuk, FannjiangStrohmer, HermanStrohmer, Rauhut, StoicaBabu, StoicaBabuLi,FastStorm}.
However, this approach has some significant drawbacks.  The theoretical requirements imposed by the classical models of compressed sensing become more stringent as the grid becomes finer.
Furthermore, making the grid finer can also lead to numerical instabilities and computational bottlenecks in practice.

Despite recent successes in many empirical disciplines, the theory of superresolution imaging remains limited.  Cand\`es and Fernandes-Granada~\cite{CandesGranda} recently made an important contribution to the mathematical analysis of superresolution, demonstrating that semi-infinite optimization could be used to solve the classical Prony problem.  Their proof technique has formed the basis of several other analyses including that of Bendory~\emph{et al}~\cite{BendoryDekelFeuer} and that of Tang~\emph{et al}~\cite{Tang14}.  To better compare with our approach, we briefly describe the approach of~\cite{BendoryDekelFeuer,CandesGranda,Tang14} here.
The classical compressed sensing arguments guarantee recovery of a sparse signal by constructing a {\em dual certificate}~\cite{CandesRombergTao,Fuchs}.   In the continuous setting of superresolution, this amounts to constructing a {\em dual polynomial}: a function of the form $Q(t) = \int \psf(s,t) q(s) dP(s)$ satisfying 
\begin{align*}
|Q(t)|&\le 1\\
|Q(t_i)| & = 1, \quad i =1 ,\ldots,M.
\end{align*}
They construct $q$ as a linear combination of $\psf(s,t_i)$ and $\frac{d}{dt_i} \psf(s,t_i)$.  In particular, they define the coefficients of this linear combination as the least squares solution to the system of equations
\begin{equation}
\label{EqnQ}
\begin{aligned}
Q(t_i) &= \text{sign}(c_i), &&\quad  i =1 ,\ldots,M\\
\frac{d}{dt}Q(t)\Big |_{t=t_i}  &= 0, && \quad i = 1,\ldots,M.
\end{aligned}
\end{equation}
They prove that, under a minimum separation condition on the $t_i$, the system has a unique solution because the matrix for the system is a perturbation of the identity, hence invertible.   

Much of the mathematical analysis on superresolution has relied heavily on the assumption that the point sources are separated by more than some minimum amount
~\cite{BatenkovYomdin,BendoryDekelFeuer,CandesGranda,DonohoStark,Eckhoff,Stable}.  
We note that in practical situations with noisy observations, some form of minimum separation may be necessary.  One can expect, however, that the required minimum separation should go to zero as the noise level decreases: a property that is not manifest in previous results.  Our approach, by contrast, does away with the minimum separation condition by observing that this matrix need not be close to the identity to be invertible.  Instead, we impose Conditions~\ref{main-conditions}~to guarantee invertibility directly.    Not surprisingly, we use techniques from T-systems to construct an analog of the polynomial $Q$ in~\eqref{EqnQ} for our specific problem.

Another key difference is that we consider the weighed objective $\int w(t)d\mu(t)$, while prior work~\cite{BendoryDekelFeuer,CandesGranda,Tang14} has analyzed the unweighted objective $\int d\mu(t)$. We, too, could not remove the separation condition without reweighing by $w(t)$.
In Section~\ref{SecSym} we provide evidence that this mathematically motivated reweighing step actually improves performance in practice.  Weighting has proven to be a powerful tool in compressed sensing, and many works have shown that weighting an $\ell_1$-like cost function can yield improved performance over standard $\ell_1$ minimization~\cite{friedlander2012recovering,khajehnejad2011analyzing,vaswani2010modified,von2007compressed}.  To our knowledge, the closest analogy to our use of weights comes from Rauhut and Ward, who use weights to balance the influence of dynamic range of bases in polynomial interpolation problems~\cite{RauhutWard}.  In the setting of this paper, weights will serve to lessen the influence of sources that have low overlap with the observed samples.

We are not the first to bring the theory of Tchebycheff systems to bear on the problem of recovering finitely supported measures.
De Castro and Gamboa~\cite{deCastroGamboa} prove that a finitely supported positive measure $\mu$ can be recovered exactly from measurements of the form 
\begin{equation*}
\label{EqnBeurling}
\Big \{\int u_0 d\mu,\ldots,\int u_nd\mu \Big \}
\end{equation*}
whenever $\{u_0,\ldots,u_n\}$ form a T-system containing the constant function \mbox{$u_0 = 1$.}
These measurements are almost identical to ours;  if we  set $u_k(t) = \psf(s_k,t)$ for $k = 1,\ldots,n$, where $\{s_1,\ldots,s_n\}=\cS$ is our measurement set, then our measurements are of the form 
$$\{x(s) \; | \; s\in \cS\} =  \Big \{ \int u_1 d\mu, \ldots, \int u_n d\mu \Big \}.$$
 However, in practice it is often  impossible to directly measure the mass $\int u_0 d\mu = \int d\mu$ as required by~\eqref{EqnBeurling}.
Moreover, the requirement that $1,\psf(s_1,t),\ldots,\psf(s_n,t)$ form a T-system does not hold for the Gaussian point spread function $\psf(s,t) = e^{-(s-t)^2}$. Therefore the theory of~\cite{deCastroGamboa} is not readily applicable to superresolution imaging.

We conclude our review of the literature by discussing some prior literature on superresolution without a minimum separation condition. The theory of signals with finite rate of innovation shows that given a superposition of pulses of the form $\sum a_k\psi(t-t_k)$, one can reconstruct the shifts $t_k$ and coefficients $a_k$ from a minimal number of samples ~\cite{DragottiVetterliBlu,Vetterli2002}. 
This holds without any separation condition on the $t_k$ and as long as the base function $\psi$ can reproduce polynomials of a certain degree see~\cite[Section A.1]{DragottiVetterliBlu} for more details. The algorithm used for this reconstruction is however based on polynomial rooting techniques that do not easily extend to higher dimensions.  
In contrast we study sparse recovery techniques which may be more stable to noise and trivially generalize to higher dimensions (although our analysis currently does not). Turning to $\ell_1$ recovery techniques, we would like to mention the work of Fuchs~\cite{Fuchs} in the case that the point spread function is band-limited and the samples are on a regularly-spaced grid. 
This result also does not require a minimum separation condition. However, our results hold for considerably more general point spread functions and sampling patterns. 
Finally, in a recent paper Bendory~\cite{BendoryDiscrete} presents an analysis of $\ell_1$ minimization in a discrete setup by imposing a Rayleigh regularity condition which, in the absence of noise, requires no minimum separation. Our results are of a different flavor, as our setup is continuous. Furthermore we require linear sample complexity while the theory of Bendory~\cite{BendoryDiscrete} requires infinitely many samples.


\section{Proofs}

In this section we prove Theorem~\ref{ThmMain} and Theorem~\ref{ThmGauss}. We start by giving a short list of notation to be used throughout the proofs.

\paragraph{Notation Glossary}

\begin{itemize}
\item We denote the inner product of 
$f,g\in \Lpspace{\dist}$ by $\Lprod{f,g}{\dist} \defn \int f(t) g(t) d\dist(t).$
\item For any differentiable function $f: \bbR^2 \to \bbR$, we denote the derivative in its first argument by $\dx f$ and in its second argument by $\dy f$.
\item For $t \in \bbR$, let $\psf_t(\cdot) = \psf(\cdot,t)$.  
\item Finally, let $K_\obs(t,\tau) = \Lprod{\psf_t, \psf_\tau}{P} = \int \psf(s,t) \psf(s,\tau) dP(s)$.
\end{itemize}

\subsection{Proof of Theorem~\ref{ThmMain}}\label{Sec4}
We prove Theorem \ref{ThmMain} in two steps. 
Firstly, we reduce the proof to constructing a function $q$ such that $\Lprod{q,\psf_t}{P}$ possesses some specific properties.  
\begin{proposition}\label{Prop:DualCert}
\label{PropRecover}
If the first three items of Conditions~\ref{main-conditions}~hold, and if there exists a function $q$ such that
\mbox{$\dualfn(t) := \Lprod{q,\psf_{t}}{\dist}$} satisfies
\begin{align}
\label{EqnDualConstraints}
\dualfn(t_j) &= w(t_j), \quad j = 1,\ldots,\numc \\
\dualfn(t) &< w(t_j), \quad \text{for } t \in [-T,T] \text{ and } t \ne t_j,\notag
\end{align}
then the true measure $\muopt \defn \sum_{j=1}^\numc \coeff_j \delta_{t_j}$ is the unique optimal solution of the program~\ref{Primal}.
\end{proposition}
This proof technique is somewhat standard \cite{CandesRombergTao,Fuchs}: the function $Q(t)$ is called a {\em dual certificate} of optimality. However, introducing the function $w(t)$ is a novel aspect of our proof.  The majority of arguments have $w(t) = 1$.  Note that when $\int \psf(s,t) dP(s)$ is independent of $t$, then $w(t)$ is a constant and we recover the usual method of proof. 

In the second step we construct $q(s)$ as a linear combination of the $t_i$-centered point spread functions $\psi(s, t_i)$ and their derivatives $\dy\psi(s, t_i)$.
 \begin{theorem} \label{ThmPop}
 Under the~Conditions~\ref{main-conditions}, there exist $\alpha_1,\ldots,\alpha_M,\beta_1,\ldots,\beta_M,c\in \bbR$ such that $Q(t) = \langle q, \psi_t\rangle_P$  satisfies \eqref{EqnDualConstraints}, 
where
 $$q(s) = \sum_{i=1}^M(\alpha_i \psi(s, t_i) + \beta_i \frac{d}{dt_i} \psi(s, t_i)) + c.$$  
 \end{theorem} 
\noindent To complete the proof of Theorem~\ref{ThmMain}, it remains to prove Proposition~\ref{PropRecover} and Theorem~\ref{ThmPop}.  Their proofs can be found in Sections~\ref{SecProofPropRecover} and~\ref{SecProofThmPop} respectively. 

\subsection{Proof of Proposition~\ref{PropRecover}}
\label{SecProofPropRecover}
We show that $\muopt$ is the optimal solution of problem~\eqref{Primal} through strong duality.  
The dual of problem~\eqref{Primal} is 
\begin{equation}
\label{EqnDual}
\begin{aligned}
\maximize_q & \quad \Lprod{q,x}{\obs} \\
\textnormal{subject to} & \quad  \Lprod{q,\psf_t}{\obs} \le w(t) \quad \text{for } t\in[-T,T].
\end{aligned}
\end{equation}
Since the primal~\eqref{Primal} is equality constrained, Slater's condition naturally holds, implying strong duality. 
As a consequence, we have 
\begin{equation*}
\Lprod{ q, x }{\dist} =  \int{w(t) d\mu(t)} \iff  \text{ $q$ is dual optimal and $\mu$ is primal optimal.}
\end{equation*}
  
Suppose $q$ satisfies~\eqref{EqnDualConstraints}.  Hence $q$ is dual feasible and we have
\begin{align*}
\Lprod{q,x}{\dist} =& \sum_{j=1}^\numc c_j  \Lprod{ q ,  \psi_{t_j}}{P}  = \sum_{j=1}^M c_j Q ( t_j) \\  =& \int {w(t) d\muopt(t) }.
\end{align*}
Therefore, $q$ is dual optimal and $\muopt$ is primal optimal.
  
Next we show uniqueness. 
Suppose the primal~\eqref{Primal} has another optimal solution 
$$\hat{\mu} = \sum_{j=1}^{\hat{M}} \hat{c}_j \deltafn{\hat{t}_j}$$ 
such that $\{\hat{t}_1,\ldots,\hat{t}_{\hat M}\} \ne \{t_1,\ldots,t_M\} \defn \TSet$.
Then we have
\begin{align*}
    \Lprod{q,x}{\dist}  
     = & \sum_j \hat{c}_j \Lprod{ q, \psi_{\hat{t}_j}}{\obs} \\
     = & \sum_{ \hat{t}_j \in \TSet } \hat{c}_j Q ( \hat{t}_j) + \sum_{\hat{t}_j \notin \TSet} \hat{c}_j Q ( \hat{t}_j)\\
     < & \sum_{ \hat{t}_j \in \TSet } \hat{c}_j w ( \hat{t}_j) + \sum_{\hat{t}_j \notin \TSet} \hat{c}_j w ( \hat{t}_j) 
     =  \int{w(t) d\hat{\mu}(t) } . 
  \end{align*}
Therefore, all optimal solutions must be supported on $\{t_1, \ldots, t_\numc \}$.

We now show that the coefficients of any optimal $\hat \mu$ are uniquely determined. By condition \CondPSD  \quad  the matrix $\int v(s) v(s)^T dP(s)$ is invertible. Since it is also positive semidefinite, then it is positive definite, so, in particular its upper $M\times M$ block is also positive definite. 
$$\det \int  \begin{bmatrix} \psf(s,t_1) \\ \vdots \\ \psf(s,t_M) \end{bmatrix}  \begin{bmatrix} \psf(s,t_1) & \ldots & \psf(s,t_M) \end{bmatrix} dP(s) \ne 0.$$  Hence there must be $s_1,\ldots, s_{M} \in \cS$ such that the matrix with entries $\psf(s_i,t_j)$ is nonsingular.

Now consider some optimal $\hat \mu = \sum_{i=1}^M \hat c_i t_i$.  Since $\hat \mu$ is feasible we have
\begin{equation*}
x(s_j) = \sum_{i=1}^M \hat c_i \psf(s_j,t_i) =  \sum_{i=1}^M c_i \psf(s_j,t_i) \quad \text{for} \quad j = 1,\ldots,M.
\end{equation*}
Since $\psi(s_i, t_j)$ is invertible, we conclude that the coefficients $c_1,\ldots,c_M$ are unique.  
Hence $\muopt$ is the unique optimal solution of~\eqref{Primal}.

\subsection{Proof of Theorem \ref{ThmPop}}
\label{SecProofThmPop}
We construct $Q(t)$ via a limiting interpolation argument due to Krein~\cite{Krein}.  We have adapted some of our proofs (with nontrivial modifications) from the aforementioned text by Karlin and Studden~\cite{KarlinStudden}. We give reference to the specific places where we borrow from classical arguments.

In the sequel, we make frequent use of the following elementary manipulation of determinants:
\begin{lemma}\label{lem:nonsingularMatrix}
If $v_0,\ldots,v_{n}$ are vectors in $\bbR^{n}$, and $n$ is even, then 
$$ \begin{vmatrix} v_1 - v_{0} & \ldots & v_n - v_0 \end{vmatrix} = \begin{vmatrix} v_1 & \ldots & v_n & v_0 \\
1 & \ldots & 1 & 1 \end{vmatrix}.$$
\end{lemma}
\begin{proof}
By elementary manipulations, both determinants in the statement of the lemma are equal to 
\begin{align*}
\begin{vmatrix}
v_1 - v_0 & \ldots & v_n - v_0 & v_0\\
0 & \ldots & 0 & 1
\end{vmatrix}.
\end{align*}
\end{proof}

In what follows, we consider $\epsilon>0$ such that 
\begin{equation*}
t_1 -\epsilon < t_1 + \epsilon < t_2-\epsilon < t_2 + \epsilon < \cdots < t_{M}-\epsilon < t_M + \epsilon.
\end{equation*}
\begin{definition} 
\label{DefNodal}
A point $t$ is a {\em nodal} zero of a continuous function $f:\mathbb R \to \mathbb R$ if $f(t) = 0$ and $f$ changes sign at $t$. A point $t$ is a {\em non-nodal} zero if $f(t) = 0$ but $f$ does not change sign at $t$.  This distinction is illustrated in Figure~\ref{fig:nodal-zeros}.
\end{definition}
\begin{figure}
\centering
\includegraphics[width=6cm]{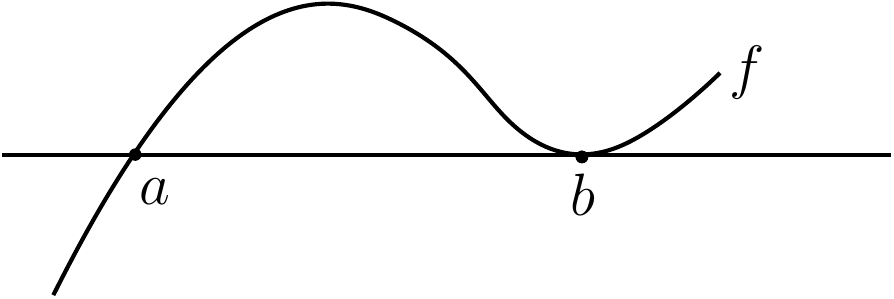}
\caption{The point $a$ is a {\em nodal} zero of $f$, and the point $b$ is a {\em non-nodal} zero of $f$.}\label{fig:nodal-zeros}
\end{figure}
Our proof of Theorem~\ref{ThmPop} proceeds as follows.  With $\epsilon$ fixed, we construct a function 
$$\Qfn_{\epsilon}(t) = \sum_{i=1}^{M} \aie \Kp(t,t_i) +  \bie \dy \Kp(t,t_i)$$ 
such that $\Qfn_{\epsilon}(t) = w(t)$ only at the points $t = t_j \pm\epsilon$ for all $j=1, 2, \ldots, M$ and the points $t_j\pm\epsilon$ are nodal zeros of $\Qfn_{\epsilon}(t) - w(t)$ for all $j=1,2,\dots, M$.
We then consider the limiting function  $\Qfn(t) = \lim\limits_{\epsilon \downarrow 0} \Qfn_\epsilon(t)$, and prove that either 
$\Qfn(t)$ satisfies \eqref{EqnDualConstraints} or \mbox{$2w(t) - \Qfn(t)$} satisfies \eqref{EqnDualConstraints}.  An illustration of this construction is pictured in Figure~\ref{fig:construction}.

\begin{figure}
\centering
\includegraphics[width=11cm]{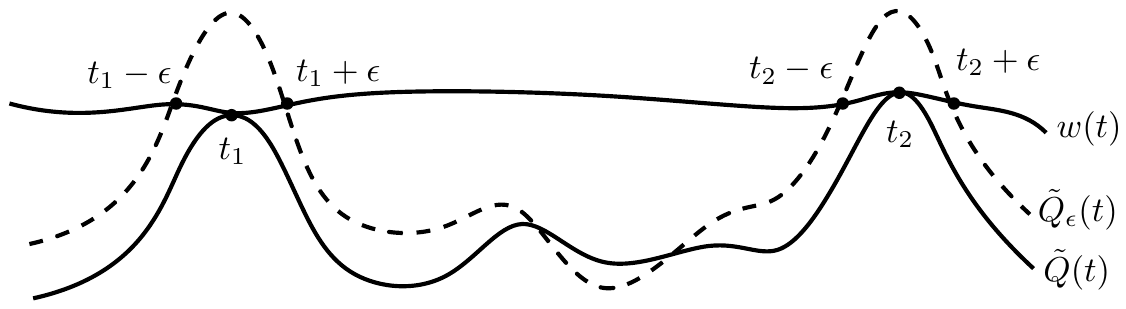}
\caption{The relationship between the functions $w(t)$, $\Qfn_\epsilon(t)$ and $\Qfn(t)$.  The function $\Qfn_\epsilon(t)$ touches $w(t)$ only at $t_i\pm \epsilon$, and these are nodal zeros of $\Qfn_\epsilon(t) - w(t)$.  The function $\Qfn(t)$ touches $w(t)$ only at $t_i$ and these are non-nodal zeros of $\Qfn(t) - w(t)$. \label{fig:construction}}
\end{figure}

We begin with the construction of $\Qfn_\epsilon$.  We aim to find the coefficients $\alpha_{\epsilon}, \beta_{\epsilon}$ to satisfy 
\begin{equation*}
\begin{aligned}
\Qfn_\epsilon(t_i - \epsilon) = w(t_i-\epsilon) \quad \text{and} \quad 
\Qfn_\epsilon(t_i + \epsilon) = w(t_i+\epsilon) \quad \text{for} \quad i = 1,\ldots,M.
\end{aligned}
\end{equation*}
This system of equations is equivalent to the system
\begin{equation}
\label{EqnQepsilon}
\begin{aligned}
\Qfn_\epsilon(t_i - \epsilon) &= w(t_i-\epsilon) \quad \text{for} \quad i = 1,\ldots,M \\
\frac{\Qfn_\epsilon(t_i + \epsilon) - \Qfn_\epsilon(t_i - \epsilon)}{2\epsilon} &= \frac{w(t_i+\epsilon) -  w(t_i-\epsilon)}{2\epsilon} \quad \text{for} \quad i = 1,\ldots,M.
\end{aligned}
\end{equation}
Note that this is a linear system of equations in $\alpha_\epsilon,\beta_\epsilon$ with coefficient matrix given by 
\begin{align*}
\mathbf{K}_\epsilon \defn \left[\begin{array}{ccc|ccc}
&&&&&\\
&K_P(t_j-\epsilon, t_i)& && \partial_2 K_P(t_j - \epsilon, t_i)&\\
&&&&&\\ \hline
&&&&&\\
& \frac1{2\epsilon}\big(K_P(t_j + \epsilon, t_i) - K_P(t_j - \epsilon, t_i)\big) &&& \frac1{2\epsilon}\big(\partial_2 K_P(t_j+\epsilon, t_i) - \partial_2 K_P(t_j - \epsilon, t_i)\big)& \\
 &&&&&
\end{array}\right].
\end{align*}
That is, the equations~\eqref{EqnQepsilon} can be written as
\begin{align*}
\mathbf K_{\epsilon}
\begin{bmatrix}
|\\
\alpha_{\epsilon}\\
|\\
|\\
\beta_{\epsilon}\\
|\\
\end{bmatrix} =  \begin{bmatrix}
w(t_1-\epsilon)\\
\vdots \\
w(t_M-\epsilon)\\
\frac1{2\epsilon}(w(t_1+\epsilon) - w(t_1 - \epsilon))\\
\vdots\\
\frac1{2\epsilon}(w(t_M + \epsilon) - w(t_M - \epsilon))
\end{bmatrix}.
\end{align*}

We first show that the matrix $\mathbf{K}_\epsilon$ is invertible for all $\epsilon$ sufficiently small. Note that as $\epsilon\to 0$ the matrix $\mathbf{K}_{\epsilon}$ converges to
\begin{align*}
{\mathbf{K}} \defn  \left[\begin{array}{ccc|ccc}
&&&&&\\
&K_P(t_j, t_i)& && \partial_2 K_P(t_j, t_i)&\\
&&&&&\\ \hline
&&&&&\\
& \partial_1 K_P(t_j, t_i) &&& \partial_1 \partial_2 K_P(t_j, t_i)& \\
 &&&&&
\end{array}\right] = \int v(s) v(s)^T dP(s),
\end{align*}
which is positive definite by {\CondPSD}. Since the entries of $\mathbf{K}_\epsilon$ converge to the entries of $\mathbf{K}$, there is a $\Delta > 0$ such that $\mathbf{K}_\epsilon$ is invertible for all $\epsilon \in (0, \Delta)$. Moreover, $\mathbf K_{\epsilon}^{-1}$ converges to $\mathbf K^{-1}$ as $\epsilon\to 0$
 and for all $\epsilon < \Delta$, the coefficients are uniquely defined as
\begin{align}\label{EqnConverge}
\begin{bmatrix}
|\\
\alpha_{\epsilon}\\
|\\
|\\
\beta_{\epsilon}\\
|\\
\end{bmatrix} = \mathbf K_{\epsilon}^{-1}  \begin{bmatrix}
w(t_1-\epsilon)\\
\vdots \\
w(t_M-\epsilon)\\
\frac1{2\epsilon}(w(t_1+\epsilon) - w(t_1 - \epsilon))\\
\vdots\\
\frac1{2\epsilon}(w(t_M + \epsilon) - w(t_M - \epsilon))
\end{bmatrix}.
\end{align}
We denote the corresponding function by 
$$\Qfn_{\epsilon}(t) \defn \sum_{i=1}^{M} \aie \Kp(t,t_i) +  \bie \dy\Kp(t,t_i).$$
Before we construct $\Qfn(t)$, we take a moment to establish the following remarkable consequences of the~{\CondOnesMatrix} condition. For all $\epsilon>0$ sufficiently small the following hold:
\begin{enumerate}[(a).]
\item $\Qfn_{\epsilon}(t) = w(t)$ only at the points $t_1 - \epsilon,t_1+\epsilon, \ldots, t_{M}-\epsilon,t_M+\epsilon$.
\item These points $t_1 - \epsilon,t_1+\epsilon, \ldots, t_{M}-\epsilon,t_M+\epsilon$ are nodal zeros of $\Qfn_{\epsilon}(t)-w(t)$.
\end{enumerate}
We adapted the proofs of (a) and (b) (with nontrivial modification) from the proofs of Theorem 1.6.1 and Theorem 1.6.2 of~\cite{KarlinStudden}.

Proof of $(a).$   Suppose for the sake of contradiction that there is a $\tau \in [-T,T]$ such that $\Qfn_{\epsilon}(\tau) = w(\tau)$ and $\tau \notin \{t_1 - \epsilon,t_1 +\epsilon, \ldots,t_M-\epsilon, t_M + \epsilon\}$.  
Then we have the system of $2M$ linear equations
\begin{equation*}
\begin{aligned}
\frac{\Qfn_{\epsilon}(t_j-\epsilon)}{w(t_j-\epsilon)} - \frac{\Qfn_{\epsilon}(\tau)}{w(\tau)} &= 0 \quad j = 1,\ldots,M \\
\frac{\Qfn_{\epsilon}(t_j + \epsilon)}{w(t_j + \epsilon)} - \frac{\Qfn_{\epsilon}(\tau)}{w(\tau)} &= 0 \quad j = 1,\ldots,M.
\end{aligned}
\end{equation*}
Rewriting this in matrix form, the coefficient vector 
$\begin{bmatrix} \alpha_\epsilon & \beta_\epsilon \end{bmatrix} = \begin{bmatrix}\alpha_\epsilon^{[1]}& \cdots&\alpha_\epsilon^{[M]}& \beta_\epsilon^{[1]}&\cdots &\beta_\epsilon^{[M]}\end{bmatrix}$ of $\Qfn_{\epsilon}$ 
satisfies
\begin{equation}
\label{Kdifferenced}
\begin{bmatrix} \alpha_\epsilon & \beta_\epsilon \end{bmatrix}
\Big (\begin{matrix}
\Kvec(t_1-\epsilon) -\Kvec(\tau) & \Kvec(t_1+\epsilon) -\Kvec(\tau)  & \ldots & \Kvec(t_M+\epsilon) -\Kvec(\tau)
\end{matrix} \Big ) =\begin{bmatrix}0 & \ldots & 0 \end{bmatrix}.
\end{equation}
By Lemma~\ref{lem:nonsingularMatrix} applied to the $2M+1$ vectors $v_1 = \Kvec(t_1 -  \epsilon), \ldots, v_{2M} = \Kvec(t_M+\epsilon)$, and $v_0 = \Kvec(\tau)$, the matrix for the system of equations~\eqref{Kdifferenced} is nonsingular if and only if the following matrix is nonsingular:
$$\begin{bmatrix} \Kvec(t_1 - \epsilon) & \ldots & \Kvec(t_M + \epsilon) &\Kvec(\tau) \\ 1 & \ldots & 1 & 1 \end{bmatrix} = \Amat(t_1 - \epsilon, \ldots, t_M+\epsilon, \tau).$$
However, this is nonsingular by the \CondOnesMatrix~ condition. This gives us the contradiction that completes the proof of part $(a)$.

\smallskip

Proof of $(b)$. Suppose for the sake of contradiction that $\Qfn_{\epsilon}(t) - w(t)$ has $\numNod < 2M$ nodal zeros and $\numNnod = 2M - \numNod$ non-nodal zeros.  Denote the nodal zeros by $\{\nodal_1, ..., \nodal_{\numNod}\}$, and denote the non-nodal zeros by $\nnodal_1, \ldots, \nnodal_{\numNnod}$.  In what follows, we obtain a contradiction by doubling the non-nodal zeros of $\Qfn_\epsilon(t) - w(t)$.  We do this by constructing a certain generalized polynomial $u(t)$ and adding a small multiple of it to $\Qfn_\epsilon(t) - w(t)$.

We divide the non-nodal zeros into groups according to whether $\Qfn_{\epsilon}(t) - w(t)$ is positive or negative in a small neighborhood around the zero;  define
$$\cI^- := \{ i \; | \; \Qfn_\epsilon \le w \text{ near } \nnodal_i\} \quad \text{and} \quad \cI^+ := \{ i \; | \; \Qfn_\epsilon \ge w \text{ near } \nnodal_i\}.$$
We first show that there are coefficients $a_0, \ldots, a_M$, and $b_1,\ldots, b_M$ such that the polynomial 
$$u(t) = \sum_{i=1}^{M} a_i \Kp(t,t_i) +  \sum_{i=1}^M b_i \dy\Kp(t,t_i) + a_0 w(t)$$ 
satisfies the system of equations
\begin{equation}
\label{System}
\begin{aligned}
u(\nnodal_j) &= +1 \quad  j \in \mathcal{I}^-\\
u(\nnodal_j) &= -1 \quad j \in \mathcal{I}^+ \\
u(\nodal_i) &= 0 \quad i = 1,\ldots,\numNod\\
u(\extra) & =0,
\end{aligned}
\end{equation}
where $\extra$ is some arbitrary additional point.  
The matrix for this system is 
\begin{equation*}
\mathbf{W} 
\begin{pmatrix}
\Kvec(\nnodal_1)^T & 1 \\
\vdots &\\
\Kvec(\nnodal_{\numNnod})^T & 1\\
\Kvec(\nodal_1)^T & 1 \\
\vdots\\
\Kvec(\nodal_{\numNod})^T & 1\\
\Kvec(\extra)&1
\end{pmatrix} 
\end{equation*}
where $\mathbf{W} = \text{diag}\big(w(\nnodal_1),\ldots,w(\nnodal_{\numNnod}),w(\nodal_1),\ldots,w(\nodal_{\numNod}),w(\extra)\big)$.
This matrix is invertible by \CondOnesMatrix~since the nodal and non-nodal zeros of $\Qfn_\epsilon(t) - w(t)$ are given by $t_1-\epsilon,\ldots,t_M+\epsilon$.
Hence there is a solution to the system~\eqref{System}.

\begin{figure}
\centering
\includegraphics[width=15cm]{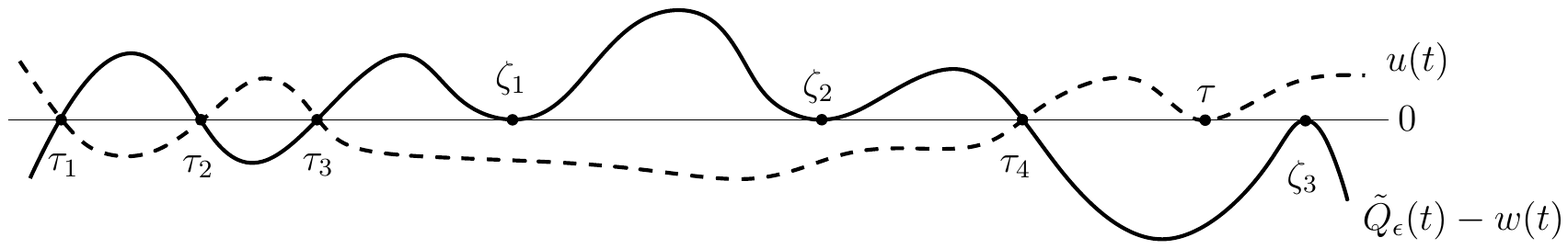}
\caption{The points $\{\tau_1,\tau_2,\tau_3,\tau_4\}$ are nodal zeros of $ \Qfn_\epsilon(t) - w(t)$, and the points $\{\zeta_1,\zeta_2,\zeta_3\}$ are non-nodal zeros.  The function $u(t)$ has the appropriate sign so that $ \Qfn_\epsilon(t) - w(t) + \delta u(t)$ retains nodal zeros at $\tau_i$, and obtains two zeros in the vicinity of each $\zeta_i$.}
\end{figure}

Now consider the function
$$U^{{\delta}}(t) = \Qfn_{\epsilon}(t) + {\delta} u(t) = \sum_{i=1}^{M} [\aie + {\delta} a_i] \Kp(t,t_i) + \sum_{i=1}^M[\bie + {\delta} b_i] \dy\Kp(t,t_i) + \delta a_0 w(t)$$
where ${\delta} > 0$.  By construction, $u(\nodal_i) = 0$, so $U^\delta(t) - w(t)$ has nodal zeros at $\nodal_1,\ldots,\nodal_{\numNod}$.  We can choose $\delta$ small enough so that $U^\delta(t) - w(t)$ vanishes twice in the vicinity of each $\nnodal_i$.  This means that $U^{\delta} (t) - w(t)$ has $2M + \numNnod$ zeros.  
Assuming $\numNnod>0$, select a subset of these zeros $p_1 < \ldots < p_{2M+1}$ such that there are two in each interval $[t_i - \rho,t_i +\rho]$.  This is possible if $\epsilon < \rho$ and $\delta$ is sufficiently small.  We have the system of $2M+1$ equations 
\begin{align*}\sum_{i=1}^{M} [\aie + {\delta} a_i] \Kp(p_1,t_i) + \sum_{i=1}^M [\bie + {\delta} b_i] \dy\Kp(p_1,t_i) &= (1-\delta a_0)w(\extra) \\
\vdots \\
\sum_{i=1}^{M}[\aie + {\delta} a_i] \Kp(p_{2M+1},t_i) +  \sum_{i=1}^M[\bie + {\delta} b_i] \dy\Kp(p_{2M+1},t_i) &= (1-\delta a_0)w(\extra).
\end{align*}
Subtracting the last equation from each of the first $2M$ equations, we find that  
$$
(\alpha_\epsilon^{[1]}+\delta a_1, \ldots, \beta_{\epsilon}^{[M]} + \delta b_M)
\begin{pmatrix}
\Kvec(p_1) - \Kvec(p_{2M+1}) & \ldots & \Kvec(p_{2M}) - \Kvec(p_{2M+1})
\end{pmatrix} =  (0,\ldots,0).
$$
This matrix is nonsingular by Lemma~\ref{lem:nonsingularMatrix} combined with the~{\CondOnesMatrix} condition.  This contradiction implies that $\numNnod = 0$.  This completes the proof of (b).

\smallskip

We now complete the proof by constructing $\Qfn(t)$ from $\Qfn_{\epsilon}(t)$ by sending $\epsilon \to 0$. 
Note that the coefficients $\alpha_{\epsilon}, \beta_{\epsilon}$ converge as $\epsilon\to 0$ since the right hand side of equation~\eqref{EqnConverge} converges to 
$$\mathbf K^{-1} \begin{bmatrix}
w(t_1)\\
\vdots \\
w(t_M)\\
w'(t_1)\\
\vdots\\
w'(t_M)
\end{bmatrix} = \begin{bmatrix}| \\ \alpha \\ | \\ | \\ \beta \\ | \end{bmatrix}.$$
We denote the limiting function by
\begin{equation}
\label{DualPoly}
\Qfn(t) = \sum_{i=1}^{M} \alpha_i \Kp(t,t_i) + \sum_{i=1}^M \beta_i \dy\Kp(t,t_i).
\end{equation}
We conclude that $w(t) - \Qfn(t)$ does not change sign at the $t_i$ since $w(t) - \Qfn_\epsilon(t)$ changes sign only at $t_i \pm \epsilon$.  

We now show that the limiting process does not introduce any additional zeros of $w(t) - \Qfn(t)$.
Suppose $\Qfn(t)$ does touch $w(t)$ at some $\tau_1 \in [-T,T]$ with $\tau_1 \ne t_i$ for any $i=1, ..., M$.  Since $w(t) - \Qfn(t)$ does not change sign, the points $t_1,\ldots,t_M,\tau_1$ are non-nodal zeros of $w(t) - \Qfn(t)$.  We find a contradiction by constructing a polynomial with two nodal zeros in the vicinity of each of these $M+1$ points (but possibly only one nodal zero in the vicinity of $\tau_1$ if $\tau_1=T$ or $\tau_1 = -T$).

For sufficiently small $\ddd>0$, the polynomial 
$$W_{\ddd}(t) = \Qfn(t) + \ddd w(t)$$
attains the value $w(t)$ twice in the vicinity of each $t_i$ and twice in the vicinity of $\tau_1$.  In other words there exist $p_1 < \ldots < p_{2M+2}$ such that
$W_{\ddd}(p_i) = w(p_i).$
Therefore
$$\Qfn(p_i) = (1-\ddd)w(p_i) \quad \text{for} \quad i = 1,\ldots,2M+2,$$
and so $\frac{\Qfn(p_i)}{w(p_i)} - \frac{\Qfn(p_{2M+1})}{w(p_{2M+1})} = 0$ for $i=1,2,...,2M$.  Thus, the coefficient vector for the polynomial $\Qfn(t)$ lies in the left nullspace  of the matrix
$$
\begin{pmatrix}
\Kvec(p_1) - \Kvec(p_{2M+1}) & \ldots & \Kvec(p_{2M}) - \Kvec(p_{2M+1})
\end{pmatrix}.
$$
However, this matrix is nonsingular by~Lemma~\ref{lem:nonsingularMatrix} and the {\CondOnesMatrix} condition.

Collecting our results, we have proven that $\Qfn(t) - w(t) = 0$ if and only if $t = t_i$ and that $\Qfn(t) - w(t)$ does not change sign when $t$ passes through $t_i$.  
Therefore one of the following is true
$$w(t) \ge \Qfn(t) \quad \text{or} \quad \Qfn(t) \ge w(t)$$
with equality iff $t = t_i$.  In the first case, $Q(t) = \Qfn(t)$ fulfills the prescriptions~\eqref{EqnDualConstraints} with 
$$q(t) = \sum_{i=1}^M \alpha_i \psf(s,t_i) + \beta_i \frac{d}{dt_i}\psf(s,t_i).$$  
In the second case, $Q(t) = 2w(t) - \Qfn(t)$ satisfies~\eqref{EqnDualConstraints} with 
$$q(t) = 2 - \sum_{i=1}^M \alpha_i \psf(s,t_i) + \beta_i \frac{d}{dt_i}\psf(s,t_i).$$

\subsection{Proof of Theorem~\ref{ThmGauss}}\label{Sec5}
\label{SecProofThmGauss}
{\CondMild} and~{\CondW} naturally hold for the Gaussian pointspread function $\psf(s,t) = e^{-(s-t)^2}$.
\CondPSD~ holds because $\psi(s, t_1), \ldots, \psi(s, t_M)$ together with their derivatives $\dy\psi(s, t_1), \ldots, \dy\psi(s, t_M)$ form a T-system (see for example~\cite{KarlinStudden}).
This means that for any $s_1 < \ldots < s_{2M} \in \bbR$, 
$$\big| v(s_1) \ldots v(s_{2M})\big| \ne 0,$$
and the determinant always takes the same sign.  
Therefore, by an integral version of the Cauchy-Binet formula for the determinant (cf. \cite{Karlin}), 
$$\Big | \int v(s) v(s)^T dP(s) \Big | = (2M)! \int\limits_{s_1 < \ldots < s_{2M}} \big |v(s_1) \ldots v(s_{2M}) \big | \begin{vmatrix} v(s_1)^T \\ \vdots \\ v(s_{2M})^T \end{vmatrix}dP(s_1) \ldots dP(s_{2M}) \ne 0.$$

To establish the~{\CondOnesMatrix} condition, we prove the slightly stronger statement:
\begin{equation}
\label{stronger}
 | \Amat(p_1, \ldots, p_{2M+1}) | = \left |  \int \begin{bmatrix}v(s)\\1\end{bmatrix}
\begin{bmatrix} \frac{\psf(s,p_1)}{w(p_1)} & \ldots & \frac{\psf(s,p_{2M+1})}{w(p_{2M+1})} \end{bmatrix}dP(s)\right | \ne 0
\end{equation}
for any distinct $p_1,\ldots,p_{2M+1}$.  When $p_1,\ldots,p_{2M+1}$ are restricted so that two points $p_i,p_j$ lie in each ball $(t_k-\rho,t_k +\rho)$, we recover the statement of \CondOnesMatrix.

We prove~\eqref{stronger} with the following key lemma.
\begin{lemma}\label{GaussianLebesgue}
For any $s_1< \ldots < s_{2M+1}$ and $t_1 < \ldots < t_{M}$,
$$\begin{vmatrix}
e^{-(s_1-t_1)^2} & \cdots & e^{-(s_{2M+1} - t_1)^2}\\
-(s_1-t_1)e^{-(s_1-t_1)^2} & \cdots & -(s_{2M+1} - t_1) e^{-(s_{2M+1} - t_1)^2}\\
\vdots & & \vdots\\
e^{-(s_1-t_M)^2} & \cdots & e^{-(s_{2M+1} - t_M)^2}\\
-(s_1-t_M)e^{-(s_1-t_M)^2} & \cdots & -(s_{2M+1} - t_M) e^{-(s_{2M+1} - t_M)^2}\\
1 & \cdots & 1
\end{vmatrix}\neq 0.$$
\end{lemma}
Before proving this lemma, we show how it can be used to prove~\eqref{stronger}.
By Lemma \ref{GaussianLebesgue}, we know in particular that for any $s_1<\cdots < s_{2M+1}$,
$$\det\begin{bmatrix}v(s_1) & \cdots & v(s_{2M+1})\\
1 & \cdots & 1\end{bmatrix}\neq 0$$
and is always the same sign.
Moreover, for any $s_1 < \cdots < s_{2M+1}$, and any $p_1<\ldots<p_{2M+1}$,
$$\det\begin{bmatrix} \psf(s_1,p_1) & \ldots & \psf(s_1, p_{2M+1})\\
\vdots\\
\psf(s_{2M+1},p_1) & \ldots & \psf(s_{2M+1}, p_{2M+1})\end{bmatrix} > 0.$$
Any function with this property is called {\em totally positive} and it is well known that the Gaussian kernel totally positive~\cite{KarlinStudden}.
Now, to show that~{\CondOnesMatrix} holds for the finite sampling measure~$P$, we use an integral version of the Cauchy-Binet formula for the determinant:
\begin{equation}
\begin{aligned}
    &\left |  \int \begin{bmatrix}v(s)\\1\end{bmatrix}
\begin{bmatrix} \frac{\psf(s,p_1)}{w(p_1)} & \ldots & \frac{\psf(s,p_{2M+1})}{w(p_{2M+1})} \end{bmatrix}dP(s)\right | =  \\
& \quad= (2M+1)! \int\displaylimits_{s_1< \cdots< s_{2M+1}}\begin{vmatrix}v(s_1) & \cdots & v(s_{2M+1})\\
1 & \cdots & 1\end{vmatrix}
\begin{vmatrix} \frac{\psf(s_1,p_1)}{w(p_1)} & \ldots & \frac{\psf(s_1, p_{2M+1})}{w(p_{2M+1})}\\
\vdots\\
\frac{\psf(s_{2M+1},p_1)}{w(p_1)} & \ldots & \frac{\psf(s_{2M+1}, p_{2M+1})}{w(p_{2M+1})}\end{vmatrix}
dP(s_1)\ldots dP(s_{2M+1}).
\end{aligned}
\end{equation}
The integral is nonzero since all integrands are nonzero and have the same sign.  
This proves~\eqref{stronger}.

\begin{proof}[Proof of Lemma \ref{GaussianLebesgue}]
\label{SecGaussLemmaProof}
Multiplying the $2i-1$ and $2i$-th row by $e^{t_i^2}$ and the $i$-th column by $e^{s_i^2}$, and subtracting $t_i$ times the $2i-1$-th row from the $2i$-th row, we obtain that we equivalently have to show that

$$\begin{vmatrix}
e^{s_1 t_1} & e^{s_2 t_1} & \ldots & e^{s_{2M+1} t_1} \\
s_1e^{s_1 t_1} & s_2e^{s_2 t_1} & \ldots & s_{2M+1}e^{s_k t_1} \\
e^{s_1 t_2} & e^{s_2 t_2} & \ldots & e^{s_{2M+1} t_2} \\
\vdots \\
e^{s_1 t_{M}} & e^{s_2 t_{M}} & \ldots & e^{s_{2M+1} t_{M}} \\
s_{1}e^{s_1 t_{M}} & s_{2}e^{s_2 t_{M}} & \ldots & s_{2M+1}e^{s_{2M+1} t_{M}} \\
e^{s_1^2} & e^{s_2^2} & \ldots & e^{s_{2M+1}^2}
\end{vmatrix} \neq 0.$$
The above matrix has a vanishing determinant if and only if there exists a nonzero vector 
$$(a_1, b_1, ..., a_M, b_M, a_{M+1})$$
in its left null space. This vector has to have nonzero last coordinate since by Example 1.1.5. in \cite{KarlinStudden}, the Gaussian kernel is extended totally positive and therefore the upper $2M\times 2M$ submatrix has a nonzero determinant. Therefore, we assume that $a_{M+1} = 1$. Thus, the matrix above has a vanishing determinant if and only if the function 
$$\sum_{i=1}^M (a_i + b_i s) e^{t_i s} + e^{s^2}$$
 has at least the $2M+1$ zeros $s_1 < s_2 < ... < s_{2M+1}$.
Lemma \ref{LemZeros}, applied to $r = M$ and $d_1 = \cdots = d_M = 1$, establishes that this is impossible.   To complete the proof of Lemma~\ref{GaussianLebesgue}, it remains to state and prove Lemma~\ref{LemZeros}.
\end{proof}

\begin{lemma}
\label{LemZeros}
Let $d_1,...,d_r\in\mathbb N$. The function 
$$\phi_{d_1,...,d_r} (s) = \sum_{i=1}^r (a_{i0} + a_{i1} s + \cdots + a_{i (2d_i - 1)}s^{2d_i - 1}) e^{t_i s} + e^{s^2}$$
has at most $2(d_1 + \cdots + d_r)$ zeros.
\end{lemma}
\begin{proof}
We are going to show that $\phi_{d_1,...,d_r}(s)$ has at most $2(d_1 + \cdots + d_r)$ zeros as follows. Let 
$$g_0(s)  = \phi_{d_1,...,d_r}(s).$$
For $k=1, ..., d_1 + \cdots + d_r$, let
\begin{equation}\label{defG}
g_k(s) =\begin{cases}
 \frac{d^2}{ds^2}\big[g_{k-1}(s)e^{(-t_j + t_1 + \cdots + t_{j-1})s}\big], & \text{ if }k = d_1 + \cdots + d_{j-1}+1\text{ for some }j,\\
 \frac{d^2}{ds^2}\big[g_{k-1}(s)\big], & \text{ otherwise}.
 \end{cases}
\end{equation}
If we show that $g_{d_1 + \cdots + d_r}(s)$ has no zeros, then, $g_{d_1 + \cdots + d_r-1}(s)$ has at most two zeros has at most two zeros, counting with multiplicity. By induction, it will follow that $g_0(s)$ has at most $2(d_1 + \cdots + d_r)$ zeros, counting with multiplicity.
Note that if $d_1 + \cdots + d_{j-1} \leq k < d_1 + \cdots + d_{j-1} + d_j$, then
\begin{equation*}
\begin{aligned}
g_k(s) =& (\tilde a_{j, 2(k - d_1 + \cdots + d_{j-1})} + \cdots + \tilde a_{j, (2d_{j} -1)} s^{2d_{j} - 1 - 2(k - d_1 + \cdots + d_{j-1})}) + \\
&+  \sum_{i=j+1}^r (\tilde a_{i0} + \cdots + \tilde a_{i (2d_i - 1)} s^{2d_i - 1})e^{(t_i - (t_1 +\cdots + t_{j-1}))r} + c f_i(r)e^{r^2}
\end{aligned}
\end{equation*}
where $c>0$ is a constant and $r := s - c_i$. We are going to show that $f_i(r)$ is a sum of squares polynomial such that one of the squares is a positive constant. This would mean that $g_k(s) = f_k(s)e^{s^2}$ has no zeros.

Denote
\begin{align*}
p_0(s) &= 1\\
p_1(s) &= 2s\\
&\vdots\\
p_i(s) & = 2sp_{i-1}(s - c_i) + p_{i-1}'(s-c_i),
\end{align*}
where $c_1, ..., c_k$ are constants. It follows by induction that the degree of $p_i(s)$ is $\deg(p_i) = i$ and the leading coefficient of $p_i(s)$ is $2^i$.

We will show by induction that 
\begin{align*}
f_i(s) &= p_i(s)^2 + \frac12p_i'(s)^2 + \cdots + \frac1{2^ii!}p_i^{(i)}(s)^2\\
&= \sum_{j=0}^i \frac1{2^j j!} p_i^{(j)}(s)^2.
\end{align*}

When $i=0$, we have that $f_0(s) = 1$ and $\sum_{j=0}^0 \frac1{2^jj!}p_0^{(j)}(s)^2 = 1$. We are going to prove the general statement by induction.
Suppose the statement is true for $i-1$.  By the relationship~\eqref{defG}, we have
\begin{align}\label{f2}
f_i(s)e^{s^2} = &\frac{d^2}{ds^2}\big[e^{s^2}f_{i-1}(s-c_i)\big] = \frac{d^2}{ds^2}\big[e^{s^2}\sum_{j=0}^{i-1} \frac1{2^jj!}p_{i-1}^{(j)}(s-c_i)^2\big]\\\notag
= &\sum_{j=0}^{i-1} \frac{e^{s^2}}{2^jj!} \Big \{ 2p_{i-1}^{(j+2)}(s-c_i)p_{i-1}^{(j)}(s-c_i) + 2p_{i-1}^{(j+1)}(s-c_i)^2 \\
& + (4s^2+2)p_{i-1}^{(j)}(s-c_i)^2 + 8sp_{i-1}^{(j)}(s-c_i)p_{i-1}^{(j+1)}(s-c_i) \Big \} \notag
\end{align}
We need to show that this expression is equal to
$e^{s^2}(\sum_{j=0}^{i}  \frac{p_{i}^{(j)}(s)^2}{2^jj!} )$.
Since
$$p_i(s)  = 2sp_{i-1}(s - c_i) + p_{i-1}'(s-c_i),$$ it follows by induction that
$p_i^{(j)}(s) = 2jp_{i-1}^{(j-1)}(s-c_i) + 2sp_{i-1}^{(j)}(s-c_i) + p_{i-1}^{(j+1)}(s-c_i).$
Therefore we obtain
\begin{equation}
\begin{aligned}
\label{f1}
e^{s^2}(\sum_{j=0}^{i}  \frac{p_{i}^{(j)}(s)^2}{2^jj!} )=&e^{s^2}\sum_{j=0}^{i} \frac1{2^jj!} \Big[ 2jp_{i-1}^{(j-1)}(s-c_i) + 2sp_{i-1}^{(j)}(s-c_i) + p_{i-1}^{(j+1)}(s-c_i)\Big]^2.\\
=&e^{s^2}\sum_{j=0}^{i}\frac1{2^jj!} \Big[ 4j^2 p_{i-1}^{(j-1)}(s-c_i)^2 + 4s^2p_{i-1}^{(j)}(s-c_I)^2 + p_{i-1}^{(j+1)}(s-c_i)^2+\\ 
&\qquad +8jsp_{i-1}^{(j-1)}(s-c_i)p_{i-1}^{(j)}(s-c_i) + \\ 
&\qquad +4sp_{i-1}^{(j)}(s-c_i)p_{i-1}^{(j+1)} + 4jp_{i-1}^{(j-1)}(s-c_i)p_{i-1}^{(j+1)}(s-c_i) \Big]
\end{aligned}
\end{equation}

There are four types of terms in the sums (\ref{f2}) and (\ref{f1}):
\begin{equation*}
p_{i-1}^{(j)}(s-c_i)^2, \quad s^2p_{i-1}^{(j)}(s-c_i)^2, \quad p_{i-1}^{(j-1)}(s-c_i)p_{i-1}^{(j)}(s-c_i), \quad \text{ and } \quad  sp_{i-1}^{(j-1)}(s-c_i)p_{i-1}^{(j)}(s-c_i).
\end{equation*}
For a fixed $j\in\{0, 1, ..., i+1\}$, it is easy to check that the coefficients in front of each of these terms in (\ref{f2}) and (\ref{f1}) are equal. Therefore,
\begin{align*}
f_i(s) &= p_i(s)^2 + \frac12p_i'(s)^2 + \cdots + \frac1{2^ii!}p_i^{(i)}(s)^2\\
&= \sum_{j=0}^i \frac1{2^j j!} p_i^{(j)}(s)^2
\end{align*}
Note that since $\deg(p_i) = i$, then, $p_i^{(i)}(s)$ equals the leading coefficient of $p_i(s)$, which, as we discussed above, equals $2^i$. Therefore, the term $\frac1{2^i i!} p_i^{(i)}(s)^2 = 2^i i!$. Thus, one of the squares in $f_i(s)$ is a positive number, so $f_i(s) > 0$ for all $s$.
\end{proof}

\section{Numerical experiments}\label{SecSym}
In this section we present the results of several numerical experiments to complement our theoretical results.   
To allow for potentially noisy observations, we solve the constrained least squares problem
\begin{equation}
\label{EqnPracticalProblem}
\begin{aligned}
\underset{\mu\ge 0}{\minimize} &\quad \sum_{i=1}^n \left ( \int \psf(s_i,t) d\mu(t) - x(s_i) \right )^2\\
\text{subject to} & \quad \int w(t)\mu(dt) \le \tau
\end{aligned}
\end{equation}
using the conditional gradient method proposed in~\cite{NickGeoffBen}.

\subsection{Reweighing matters for source localization}
Our first numerical experiment provides evidence that weighting by $w(t)$ helps recover point sources near the border of the image.  
This matches our intuition: near the border, the mass of an observed point-source is smaller than if it were measured in the center of the image.  
Hence, if we didn't weight the candidate locations, sources that are very close to the edge of the image would be beneficial to add to the representation.

We simulate two populations of images, one with point sources located away from the image boundary, and one with point sources located near the image boundary.  
For each population of images, we solve \eqref{EqnPracticalProblem} with $w(t)=\int \psf(s,t) dP(s)$ (weighted) and with $w(t) = 1$ (unweighted).
We find that the solutions to~\eqref{EqnPracticalProblem} recover the true point sources more accurately with $w(t) = \int \psf(s,t) dP(s)$.

We use the same procedure for computing accuracy as in~\cite{SMIPaper}.  
Namely we match true point sources to estimated point courses and compute the {\em F-score} of the match. 
To describe this procedure in detail, we compute the F-score by solving a bipartite graph matching problem.  In particular, we form the bipartite graph with an edge between $t_i$ and $\hat t_j$ for all $i,j$ such that $\Vert t_i - \hat t_j \Vert < r$, where $r>0$ is a tolerance parameter, and $\hat t_1, \ldots, \hat t_N$  are the estimated point sources. Then we greedily select edges from this graph under the constraint that no two selected edges can share the same vertex; that is, no $t_i$ can be paired with two $\hat t_{j}, \hat t_k$ or vice versa.
Finally, the $\hat t_i$ successfully paired with some $t_j$ are categorized as true positives, and we denote their number by $T_P$.  
The number of false negatives is $F_N = M - T_P$, and the number of false positives is $N - T_P$.  The precision and recall are then $P = \frac{T_P}{T_P + F_N},$ and $R = \frac{T_P}{T_P + F_P}$ respectively, and the F-score is the harmonic mean:
$$F = \frac{2P R}{P + R}.$$
We find a match by greedily pairing points of $\{\tau_1,\ldots,\tau_N\}$ to elements of $\{t_1,\ldots,t_M\}$, and a tolerance radius $r>0$ upper bounds the allow distance between any potential pairs.  To emphasize the dependence on $r$, we sometimes write $F(r)$ for the F-score.

Both populations contain $100$ images simulated using the Gaussian point spread function $$\psf(s,t) = e^{-\frac{(s-t)^2}{\sigma^2}}$$ with $\sigma = 0.1$, and in both cases, the measurement set $\cS$ is a dense uniform grid of $n = 100$ points covering $[0,1]$.
The populations differ in how the point sources for each image are chosen.  Each image in the first population has five points drawn uniformly in the interval $(.1,.9)$, while each image in the second population has a total of four point sources with two point sources in each of the two boundary regions $(0,.1)$ and $(.9,1)$.  
In both cases we assign intensity of $1$ to all point sources, and solve \eqref{EqnPracticalProblem} using an optimal value of $\tau$ (chosen with a preliminary simulation).

The results are displayed in Figure~\ref{FigCentBdry}.  The left subplot shows that the F-scores are essentially the same for the weighted and unweighted problems when the point sources are away from the boundary.   This is not surprising because when $t$ is away from the border of the image, then $\int \psf(s,t)dP(s)$ is essentially a constant, independent of t. 
But when the point sources are near the boundary, the weighting matters and the F-scores are dramatically better.

\begin{figure}[ht]
\begin{subfigure}{.45\textwidth}
\includegraphics[height=7.7cm]{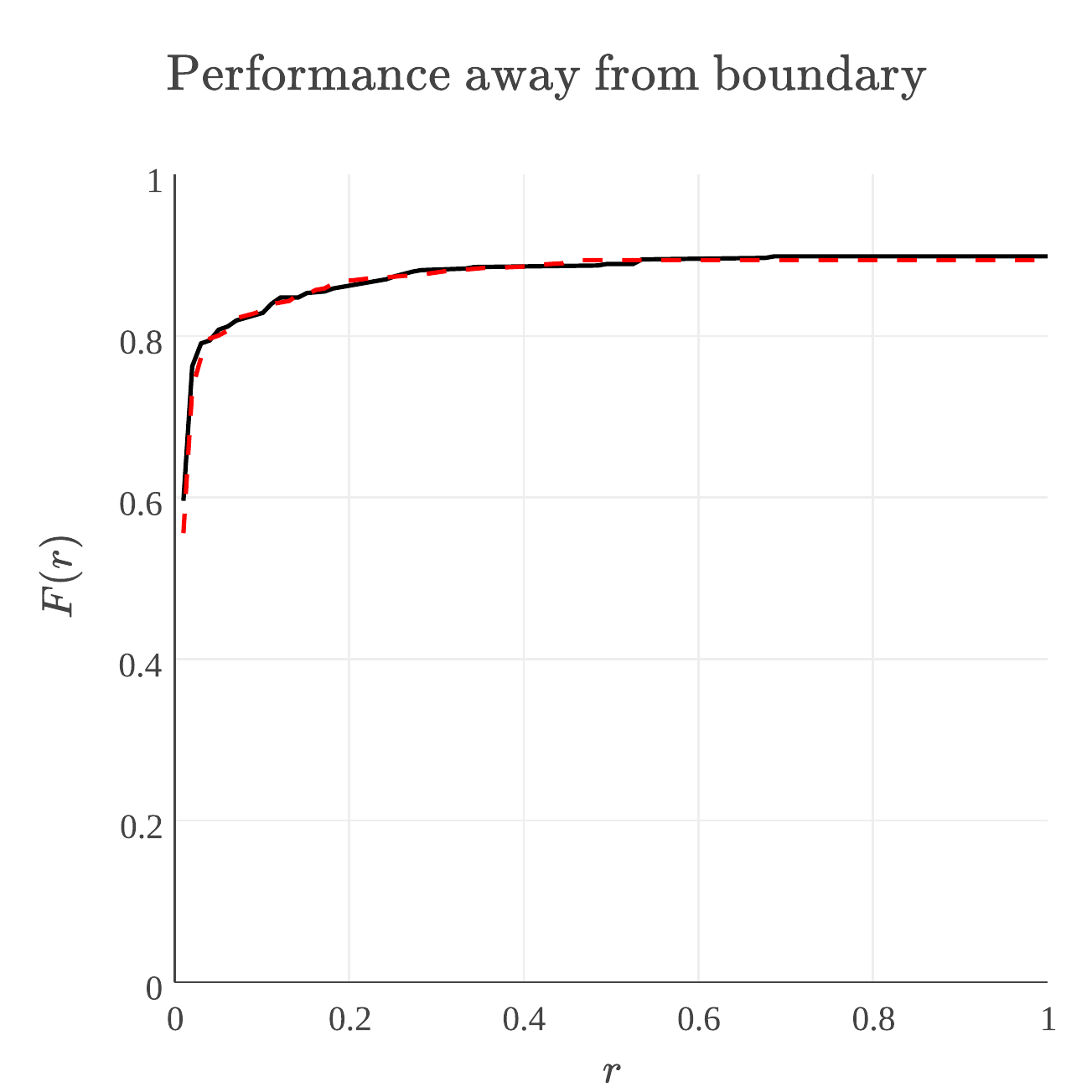}
\end{subfigure}
\begin{subfigure}{.59\textwidth}
\includegraphics[height=7.7cm]{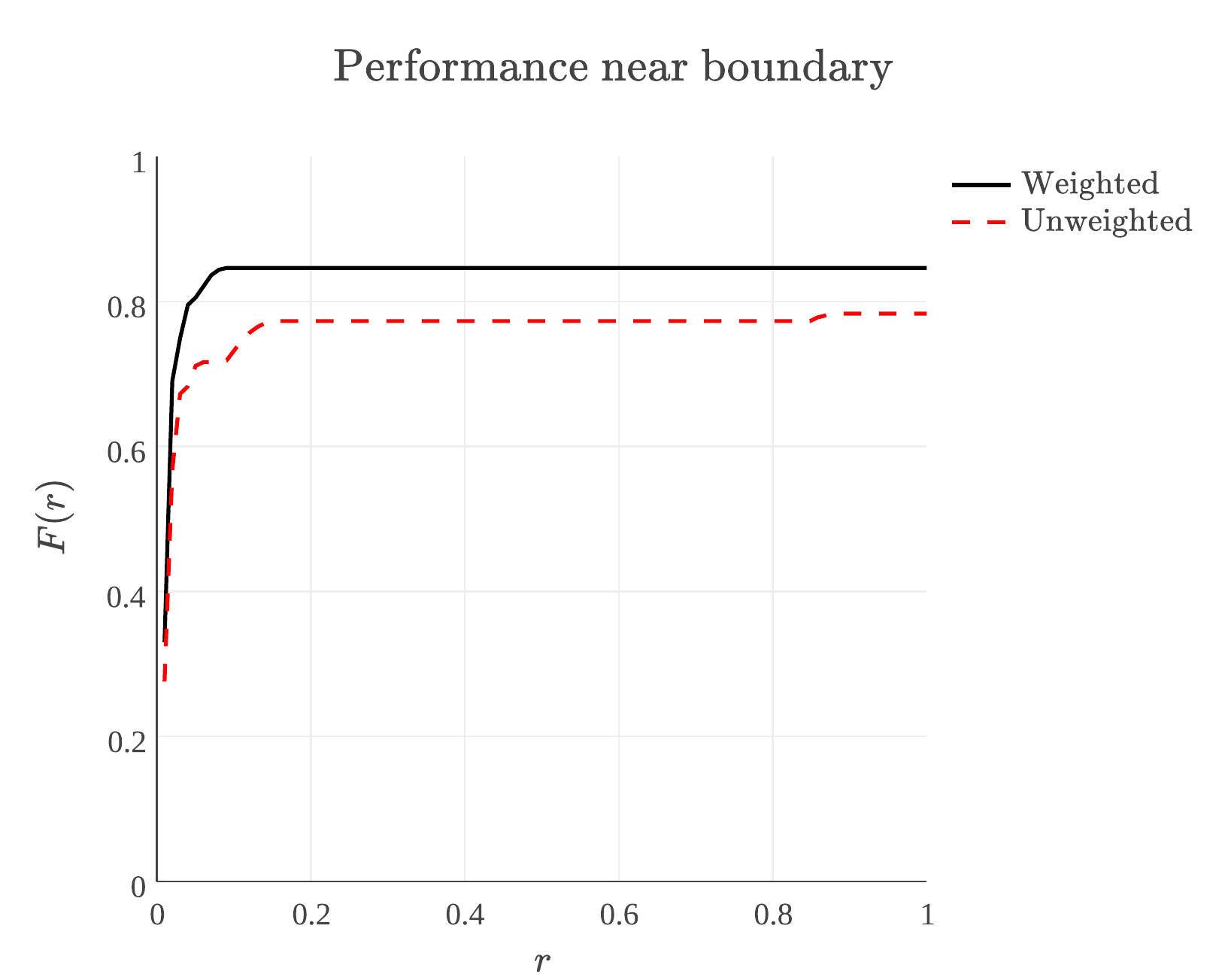}
\end{subfigure}
\caption{ {\bf Reweighing matters for source localization.} The two plots above compare the quality of solutions to the weighted problem (with $w(t) = \int \psf(s,t)dP(s)$) and the unweighted problem (with $w(t) = 1$).  
When point sources are away from the boundary (left plot), the performance is nearly identical.  But when the point sources are near the boundary (right plot), the weighted method performs significantly better.}
\label{FigCentBdry}
\end{figure}

\subsection{Sensitivity to point-source separation}
Our theoretical results assert that in the absence of noise the optimal solution of \eqref{Primal} recovers point sources with no minimum bound on the separation.
In the following experiment, we explore the ability of~\eqref{EqnPracticalProblem} to recover pairs of points as a function of their separation. 
The setup is similar to the first numerical experiment.
We use the Gaussian point spread function with $\sigma = 0.1$ as before, but here we observe only $n=50$ samples.  For each separation $d\in\{.1\sigma, .2\sigma,\ldots,1.9\sigma, 2\sigma\}$, we simulate a population of $20$ images containing two point sources separated by $d$.  The point sources are chosen by picking a random point $x$ away from the border of the image and placing two point sources at $x\pm \frac d 2$. Again, each point source is assigned an intensity of $1$, and we attempt to recover the locations of the point sources by solving~\eqref{EqnPracticalProblem}.  

\begin{figure}
\begin{subfigure}[b]{.49\textwidth}
\centering
\includegraphics[width = \textwidth]{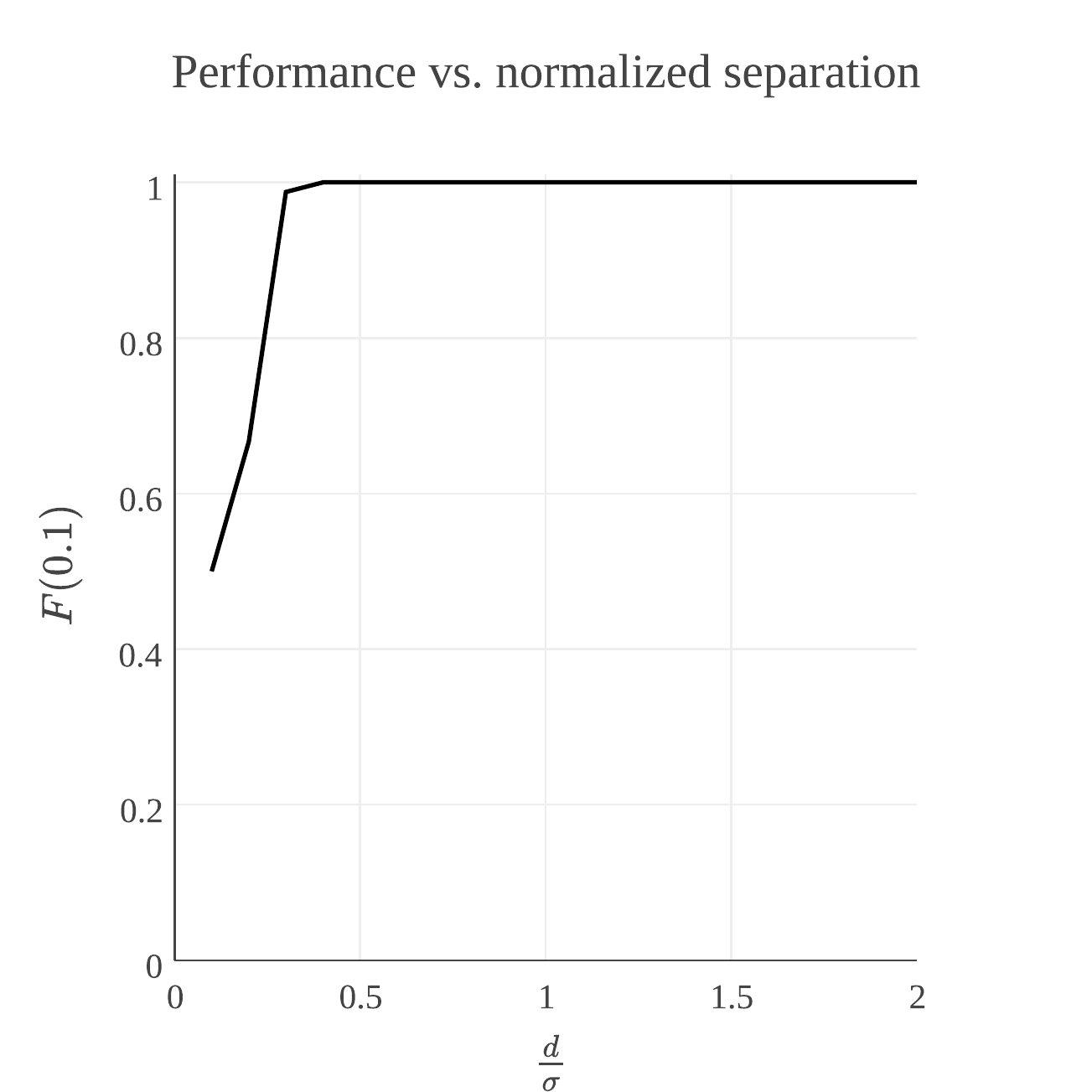}
\caption{}
\end{subfigure}
\begin{subfigure}[b]{.49\textwidth}
\centering
\includegraphics[width = \textwidth]{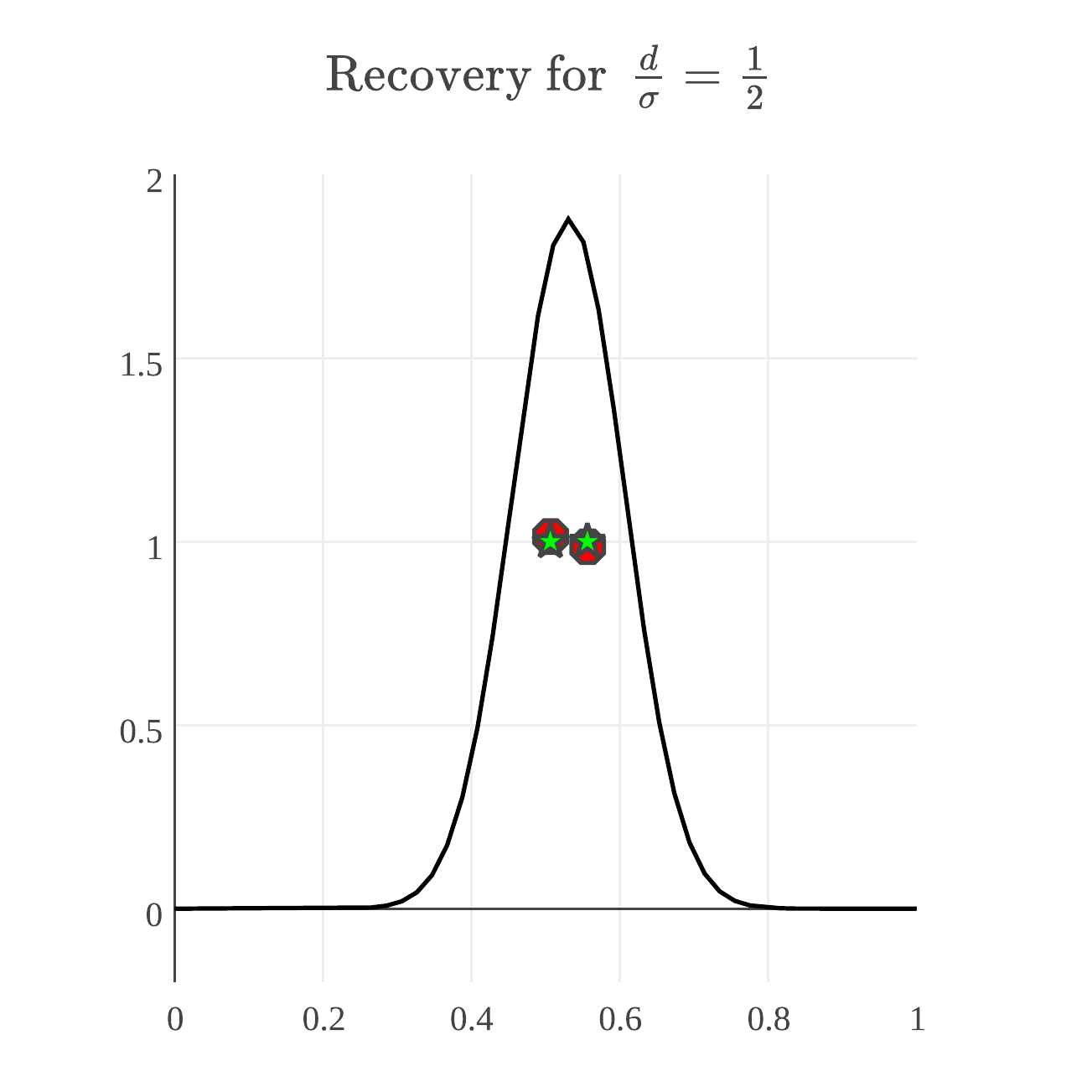}
\caption{}
\end{subfigure}
\caption{{\bf Sensitivity to point-source separation.} (a) The F-score at tolerance radius $r=0.1$ as a function of normalized separation $\frac d \sigma$.
(b) The black trace shows an image for $\frac d \sigma = \frac 1 2$.  The green stars show the locations (x-coordinate) and weights (y-coordinate) of the true point sources.  The red dots show the recovered locations and weights.}
\label{FigS}
\end{figure}

In the left subplot of Figure~\ref{FigS} we plot F-score versus separation for the value of $\tau$ that produces the best F-scores.  Note that we achieve near perfect recovery for separations greater than $\frac \sigma 4$.  The right subplot of Figure~\ref{FigS} shows the observations, true point sources, and estimated point sources for a separation of $\frac d \sigma = \frac 1 2$.  Note the near perfect recovery in spite of the small separation.

Due to numerical issues, we cannot localize point sources with arbitrarily small $d>0$.  Indeed, the F-score for $\frac d \sigma < \frac 1 4$ is quite poor.
This does not contradict our theory because  numerical ill-conditioning is in effect adding noise to the recovery problem, and we expect that a  separation condition will be necessary in the presence of noise.

\subsection{Sensitivity to noise}

Next, we investigate the performance of~\eqref{EqnPracticalProblem} in the presence of additive noise. 
The setup is identical to the previous numerical experiment, except that we add Gaussian noise to the observations.  
In particular, our noisy observations are 
$$\{ x(s_i) + \eta_i \; | \; s_i \in \cS \}$$
where $\eta_i\sim\cN(0,0.1)$. 

We measure the performance of~\eqref{EqnPracticalProblem} in Figure~\ref{FigNoise}.
Note that we achieve near-perfect recovery when $d> \sigma$.  
However, if $d< \sigma$ the F-scores are clearly worse than the noiseless case. Unsurprisingly, we observe that sources must be separated in order to recover their locations to reasonable precision.  
We defer an investigation of the dependence of the signal separation as a function of the signal-to-noise ratio to future work.

\begin{figure}
\begin{subfigure}[b]{.49\textwidth}
\centering
\includegraphics[width = \textwidth]{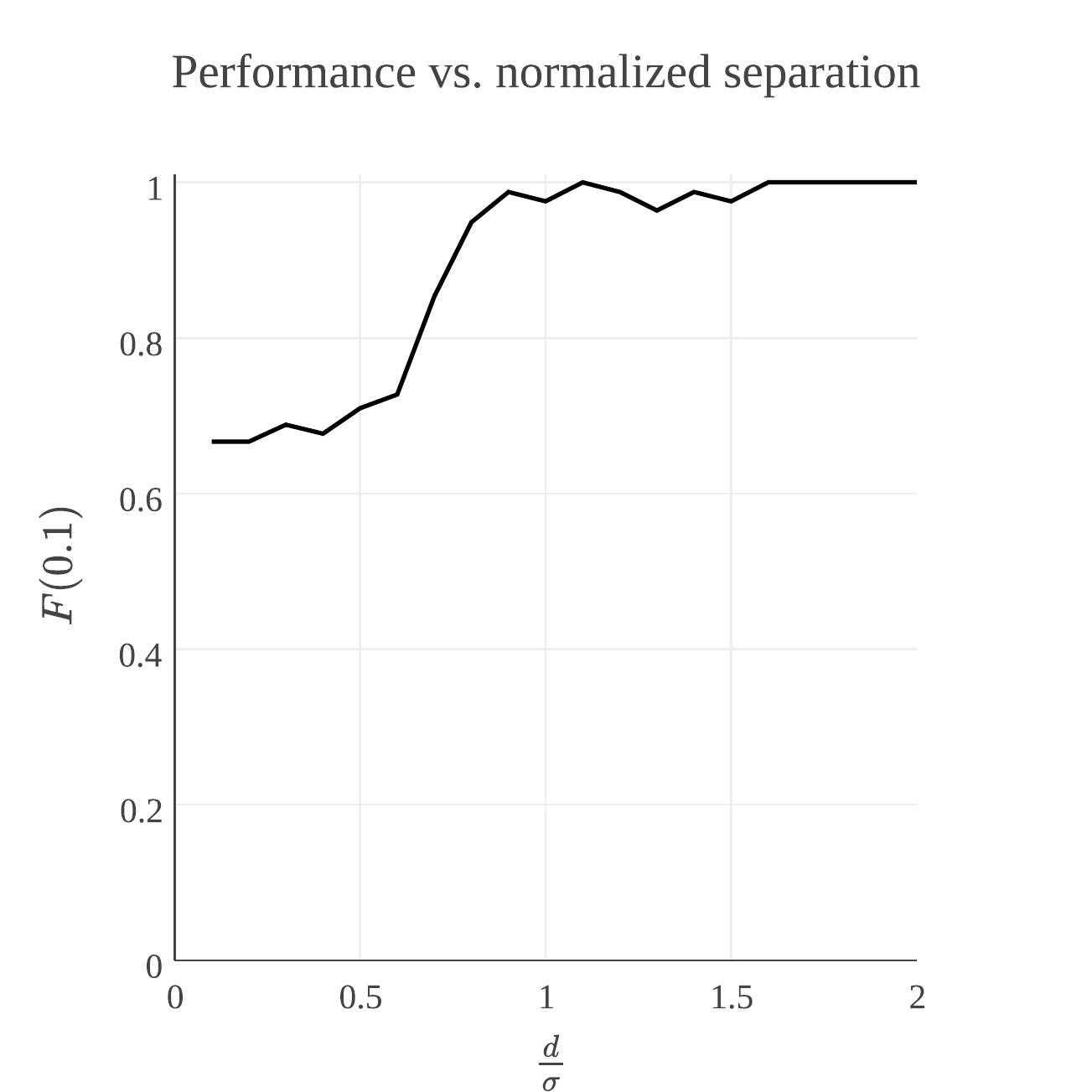}
\caption{}
\end{subfigure}
\begin{subfigure}[b]{.49\textwidth}
\centering
\includegraphics[width = \textwidth]{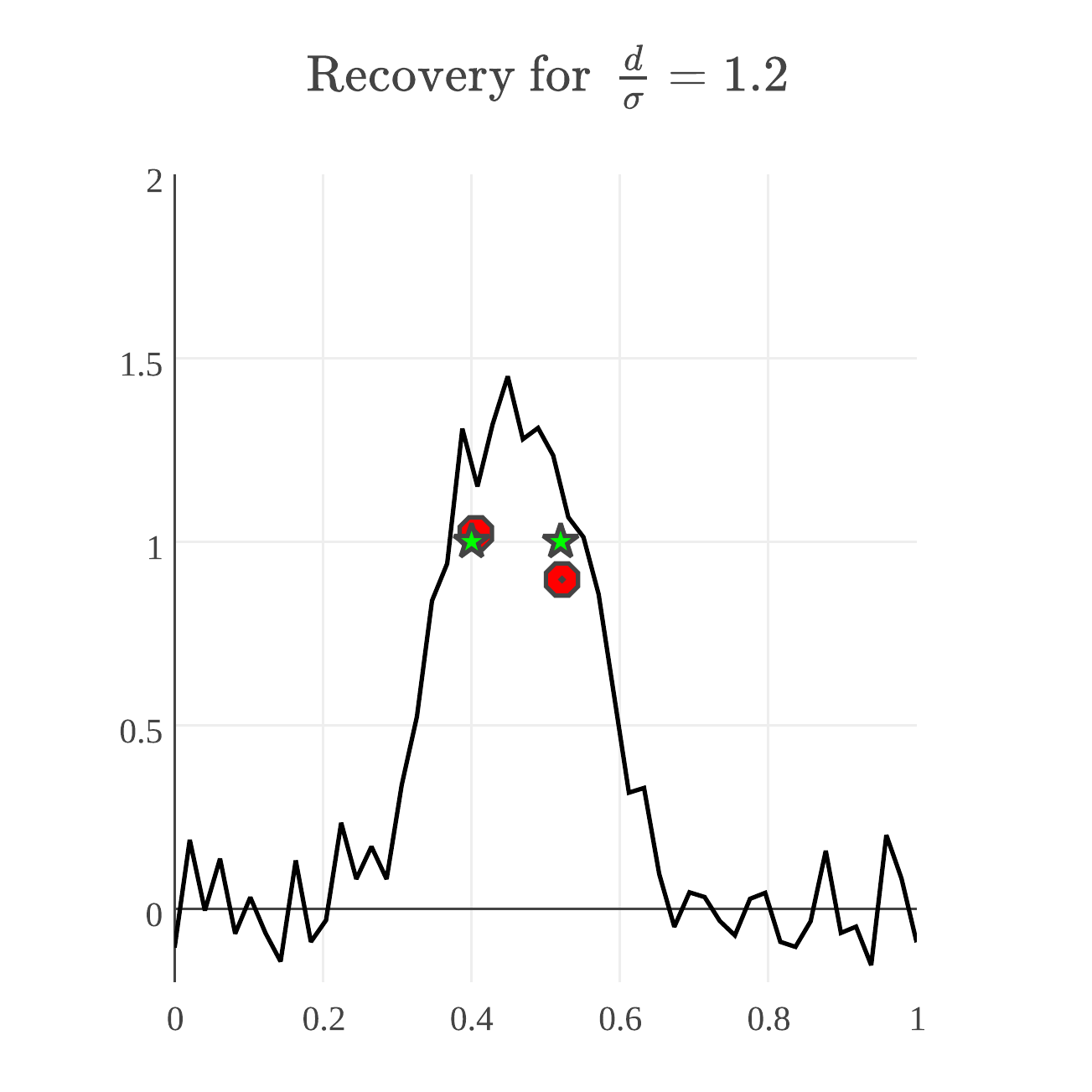}
\caption{}
\end{subfigure}
\caption{{\bf Sensitivity to noise.}  (a) The F-score at tolerance radius $r=0.1$ as a function of normalized separation $\frac d \sigma$.
(b) The black trace is the $50$ pixel image we observe.  The green stars show the locations (x-coordinate) and weights (y-coordinate) of the true point sources.   The red dots show the recovered locations and weights.}
\label{FigNoise}
\end{figure}

\subsection{Extension to two-dimensions}

Though our proof does not extend as is, we do expect generalizations of our recovery result to higher dimensional settings.  The optimization problem~\eqref{EqnPracticalProblem} extends immediately to arbitrary dimensions, and we have observed that it performs quite well in practice.  We demonstrate in Figure~\ref{FigHD} the power of applying~\eqref{EqnPracticalProblem} to a high density fluorescence image in simulation.  
Figure~\ref{FigHD} shows an image simulated with parameters specified by the Single Molecule Localization Microscopy challenge~\cite{SMIChallenge}.  In this challenge, point sources are blurred by a Gaussian point-spread function and then corrupted by noise.
The green stars show the true locations of a simulated collection of point sources, and the red dots show the support of the measure output by~\eqref{EqnPracticalProblem} applied to the greyscale image forming the background of Figure~\ref{FigHD}.  The overlap between the true locations and estimated locations is near perfect with an F-score of $0.98$ for a tolerance radius corresponding to one third of a pixel.

\begin{figure}
\centering
\includegraphics[width= \textwidth]{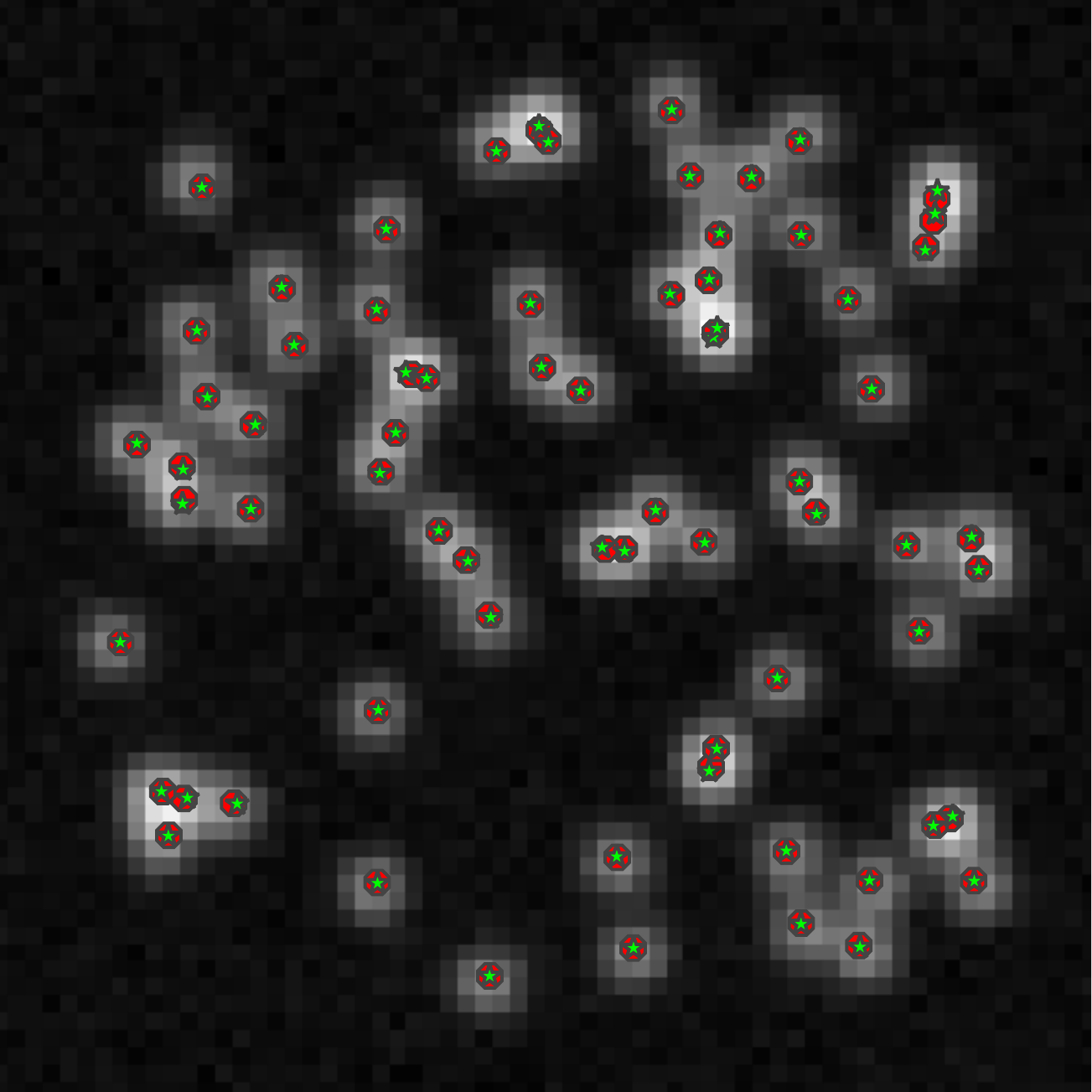}
\caption{{\bf High density single molecule imaging.}  The green stars show the locations of a simulated collection point sources, and the greyscale background shows the noisy, pixelated point spread image. The red dots show the support of the measure-valued solution of~\eqref{EqnPracticalProblem}. } 
\label{FigHD}
\end{figure}

\section{Conclusions and Future Work}

In this paper we have demonstrated that one can recover the centers of a nonnegative sum of Gaussians from a few samples by solving a convex optimization problem.  This recovery is theoretically possible no matter how close the true centers are to one-another.  We remark that similar results are true for recovering measures from their moments.  
Indeed, the atoms of a positive atomic measure can be recovered no matter how close together the atoms are, provided one observes twice the number of moments as there are atoms.
Our work can be seen as a generalization of this result, applying generalized polynomials and the theory of Tchebycheff systems in place of properties of Vandermonde systems.

As we discussed in our numerical experiments, this work opens up several theoretical  problems that would benefit from future investigation.  We close with a very brief discussion of some of the possible extensions.

\paragraph{Noise}  Motivated by the fact that there is no separation condition in the absence of noise, it would be interesting to study how the required separation decays to zero as the noise level decreases.  One of the key-advantages of using convex optimization for signal processing is that dual certificates generically give stability results, in the same way that Lagrange multipliers measure sensitivity in linear programming.   
Previous work on estimating line-spectra has shown that dual polynomials constructed for noiseless recovery extend to certify properties of estimation and localization in the presence of noise~\cite{cg_noisy,granda2, Tang14}.  We believe that these methods should be directly applicable to our problem set-up.

\paragraph{Higher dimensions} One logical extension is proving that the same results hold in higher dimensions. Most scientific and engineering applications of interest have point sources arising one to four dimensions, and we expect that some version of our results should hold in higher dimensions.  Indeed, we believe a guarantee for recovery with no separation condition can be proven in higher dimensions with noiseless observations.  However, it is not straightforward to extend our results to higher dimensions because the theory of Tchebycheff systems is only developed in one dimension.  In particular, our approach using limits of polynomials does not directly generalize to higher dimensions.

\paragraph{Other point spread functions} We have shown that our Conditions~\ref{main-conditions}~hold for the Gaussian point spread function, which is commonly used in microscopy as an approximation to an Airy function.  It will be very useful to show that they also hold for other point spread functions such as the Airy function and other common physical models.  Our proof relied heavily on algebraic properties of the Gaussian, but there is a long, rich history of determinantal systems that may apply to generalize our result.  In particular, works on properties of totally positive systems may be fruitful for such generalizations~\cite{ando1987totally,pinkus2010totally}.

\paragraph{Model mismatch in the point spread function}  Our analysis relies on perfect knowledge of the point spread function.  In practice one never has an exact analytic expression for the point spread function.  Aberrations in manufacturing and scattering media can lead to distortions in the image not properly captured by a forward model.  It would be interesting to derive guarantees on recovery that assume only partial knowledge of the point spread function.  Note that the optimization problem of searching both for the locations of the sources and for the associated wave-function is a blind deconvolution problem, and techniques from this well-studied problem could likely be extended to the super-resolution setting.  If successful, such methods could have immediate practical impact when applied to denoising images in molecular, cellular, and astronomical imaging.
 
 \section*{Acknowledgements}
We would like to thank Pablo Parrilo for introducing us to the theory of T-systems.  We would also like to thank Nicholas Boyd, Mahdi Soltanolkotabi, Bernd Sturmfels, and Gongguo Tang for many useful conversations about this work.

BR is generously supported by ONR awards N00014-11-1-0723 and N00014-13-1-0129, NSF awards CCF-1148243 and CCF-1217058, AFOSR award FA9550-13-1-0138, and a Sloan Research Fellowship.  GS was generously supported by NSF award CCF-1148243.  This research is supported in part by NSF CISE Expeditions Award CCF-1139158, LBNL Award 7076018, and DARPA XData Award FA8750-12-2-0331, and gifts from Amazon Web Services, Google, SAP, The Thomas and Stacey Siebel Foundation, Adatao, Adobe, Apple, Inc., Blue Goji, Bosch, C3Energy, Cisco, Cray, Cloudera, EMC2, Ericsson, Facebook, Guavus, HP, Huawei, Informatica, Intel, Microsoft, NetApp, Pivotal, Samsung, Schlumberger, Splunk, Virdata and VMware.
 
 \bibliography{superRes}{}
\bibliographystyle{plain}

\end{document}